\documentclass[a4paper,12pt,leqno]{amsart}
\usepackage{latexsym}
\usepackage{amsmath}
\usepackage{amssymb}
\usepackage{amsthm}
\usepackage{indentfirst}
\usepackage[shortlabels]{enumitem}
\usepackage[english]{babel} 
\usepackage[T1]{fontenc}
\usepackage{breqn}
\usepackage{cite}    
\usepackage{mathrsfs}                   
\usepackage[colorlinks=true]{hyperref}

\theoremstyle{plain} 
\newtheorem{thm}{Theorem}[section] 
\newtheorem{cor}[thm]{Corollary} 
\newtheorem{lem}[thm]{Lemma} 
\newtheorem{prop}[thm]{Proposition} 
\newtheorem{rmk}[thm]{Remark}

\newtheorem{claim}[thm]{Claim}

\theoremstyle{definition}

\numberwithin{equation}{section}

\author[D.~Mastrostefano]{Daniele Mastrostefano}
\address{University of Warwick, Mathematics Institute, Zeeman Building, Coventry, CV4 7AL, UK}
\email{Daniele.Mastrostefano@warwick.ac.uk}

\keywords{Multiplicative and additive functions in arithmetic progressions; partial sums with Ramanujan sums; integrals of exponential sums over minor arcs.}
\subjclass[2010]{Primary: 11N64. Secondary: 11B25, 11N37, 11L99.}

\begin{document}
\title[The variance in APs of multiplicative functions close to $1$]{A lower bound for the variance in arithmetic progressions of some multiplicative functions close to $1$}\thanks{The author is funded by a Departmental Award and by an EPSRC Doctoral Training Partnership Award. The present work was carried out when the author was a second year PhD student at the University of Warwick.}

\begin{abstract}
We investigate lower bounds for the variance in arithmetic progressions of certain multiplicative functions ``close'' to $1$. Specifically, we consider $\alpha_N$-fold divisor functions, when $\alpha_N$ is a sequence of positive real numbers approaching $1$ in a suitable way or $\alpha_N=1$, and the indicator of $y$-smooth numbers, for suitably large parameters $y$. 

As a corollary, we will strengthen a previous author's result on the first subject and obtain matching lower bounds to some Barban--Davenport--Halberstam type
theorems for $y$-smooth numbers.

Incidentally, we will also find a lower bound for the variance in arithmetic progressions of the prime factors counting functions $\omega(n)$ and $\Omega(n)$.
\end{abstract}

\maketitle

\section{Introduction}
For any complex arithmetic function $f:\mathbb{N}\rightarrow \mathbb{C}$ and any positive integer $N$ we indicate with 
$$V(N,Q; f):=\sum_{q\leq Q}\sum_{h|q}\sum_{\substack{a\ \bmod{q}\\ (a,q)=h}}\bigg|\sum_{\substack{n\leq N\\ n\equiv a\ \bmod{q}}}f(n)-\frac{1}{\varphi(q/h)}\sum_{\substack{n\leq N\\ (n,q)=h}}f(n)\bigg|^{2}$$
the probabilistic variance of $f$ in arithmetic progressions. Here $\varphi(\cdot)$ is the Euler totient function, $1<Q\leq N$ is a real number and the symbol $(\cdot,\cdot)$ stands for the greatest common divisor of two positive integers.

In a previous paper \cite{M}, the author found a lower bound for the quantity $V(N,Q;f)$ over a large class of multiplicative functions $f$, referred to as ``generalized divisor functions'', which contains, as a particular instance, all the $\alpha$--fold divisor functions $d_\alpha(n)$, for parameters $\alpha\in \mathbb{C}\setminus \{\{1\}\cup -\mathbb{N}\}$. For all such values $\alpha$, it was proved that:
\begin{equation}
\label{lowerboundgenvariance}
V(N,Q;d_\alpha)\gg_{\alpha,\delta} Q\sum_{n\leq N}|d_\alpha(n)|^2,
\end{equation} 
uniformly on $N^{1/2+\delta}\leq Q\leq N$, whenever $\delta>0$ is sufficiently small and $N$ is large enough with respect to $\alpha$ and $\delta$ (see \cite[Theorem 1.1]{M}).

When $\alpha=1$, the lower bound \eqref{lowerboundgenvariance} does not hold, as shown in \cite[Proposition 1.10]{M}, where the estimate $V(N,Q; d_1)\ll Q^2$, for any $Q\geq 1$, was proved, by an elementary direct inspection of the variance. The first new result of this paper, consequence of some new computations on a certain related mean square integral of a complete exponential sum, demonstrates that such an upper bound is sharp, at least in some ranges of $Q$.
\begin{thm}
\label{varianced1}
There exists an absolute constant $c>0$ such that 
for any $cN^{2/3}\leq Q\leq N$ and $N$ large enough, we have 
\begin{align*}
V(N,Q; d_1)\gg Q^2.
\end{align*}
\end{thm}
For parameters $\alpha=\alpha_N:=1+1/R(N)$, where $R(N)$ is a real non-vanishing function, the method developed in \cite{M}, which makes strong use of the asymptotic expansion of the partial sum of divisor functions, also produced the following result (see \cite[Theorem 1.11]{M}).
\begin{prop}
\label{propmainvariance}
Let $A>0$ and $\alpha_N$ as above with $|R(N)|\leq (\log N)^A$. Let $\delta>0$ small enough and $N^{1/2+\delta}\leq Q\leq N$. Then there exists a constant $B>0$ such that if $|R(N)|\geq B$ and $N$ is large with respect to $\delta$ and $A$, we have 
\begin{equation}
\label{varianceuniform}
V(N,Q;d_{\alpha_N})\gg_{A,\delta}\frac{Q}{R(N)^{4}}\sum_{n\leq N}d_{\alpha_N}(n)^2\gg \frac{QN}{R(N)^{4}}\exp\bigg(\bigg(2+\frac{1}{R(N)}\bigg)\frac{\log\log N}{R(N)}\bigg).
\end{equation}
\end{prop}
In particular, we notice that the lower bound \eqref{varianceuniform} is always of size $\frac{QN}{R(N)^{4}}$ whenever $|R(N)|\geq \log\log N$.
\begin{rmk}
By going through the proof of Proposition \ref{propmainvariance}, it is not difficult to verify that the same lower bound also holds when replacing the function $d_{\alpha_N}(n)$ with $\alpha_{N}^{\omega(n)}$ or with  $\alpha_{N}^{\Omega(n)}$, where $\Omega(n)$ and $\omega(n)$ stand for the prime divisors counting functions with or without multiplicity. 
\end{rmk}
In the following, we will indicate with $\varpi(n)$ the function $\omega(n)$ or $\Omega(n)$, when a statement holds for both, and with $d_{\alpha_N}^{\varpi}(n)$ the function $\alpha_{N}^{\varpi(n)}$.

The main aim of this paper is to improve, by means of a different approach, the result of Proposition \ref{propmainvariance} to what we expect to be the best possible lower bound for the variance of $d_{\alpha_N}^{\varpi}(n)$ in arithmetic progressions.
\begin{thm}
\label{thmvariancealpha1}
Let $\alpha_N=1+1/R(N)$, where $R(N)$ is a non-zero real function. Assume $N^{1/2+\delta}\leq Q\leq N$, with $\delta>0$ sufficiently small. Then there exists a constant $C=C(\delta)>0$ such that if $C\log \log N\leq |R(N)|\leq N^{\delta/12}$ and $N$ is large in terms of $\delta$, we have
\begin{equation}
\label{mainlowerbound}
V(N,Q; d_{\alpha_N}^{\varpi})\gg_{\delta}\frac{QN}{R(N)^2}\log\bigg(\frac{\log N}{\log(2N/Q)}\bigg)+Q^2.
\end{equation}
\end{thm}
Compared to \eqref{varianceuniform}, the lower bound \eqref{mainlowerbound} improves the exponent of $R(N)$, shows the presence of the extra factor $Q^2$, which dominates on certain ranges of $R(N)$, and $|R(N)|$ is allowed to grow much bigger than an arbitrarily large power of $\log N$. 

In order to exploit more the extra cancellation we have, compared to \eqref{lowerboundgenvariance}, when $\alpha$ is close to $1$, we will input the Taylor expansion of the function $d_{\alpha_N}^{\varpi}(n)=(1+1/R(N))^{\varpi(n)}$ into our new computations. Since the function $\varpi(n)$ is, for the majority of positive integers $n\leq N$, of size roughly $\log\log N$ (see e.g. \eqref{meanvarpi} below), this justifies the condition $|R(N)|\geq C\log\log N$ in the hypotheses of Theorem \ref{thmvariancealpha1}.

Regarding the additive function $\varpi(n)$ we will prove the following result.
\begin{thm}
\label{varianceomega}
Assume $N^{1/2+\delta}\leq Q\leq N$, with $\delta>0$ sufficiently small. Then we have
\begin{equation*}
V(N,Q; \varpi)\gg_{\delta} Q^2(\log\log N)^2+QN\log\bigg(\frac{\log N}{\log(2N/Q)}\bigg),
\end{equation*}
if $N$ is large enough in terms of $\delta$.
\end{thm} 
\begin{rmk}
The proof of Theorem \ref{varianceomega} contains some aspects and computations preliminary to that of Theorem \ref{thmvariancealpha1}. This is why we decided to insert such result here.
\end{rmk}
The sequence of functions $d_{\alpha_N}^{\varpi}(n)$ is only one instance of a wide class of multiplicative functions ``close'' to $1$. Another interesting representative of such class is the characteristic function of the $y$--smooth numbers, for parameters $y$ near $N$. These are defined as those numbers made only by prime factors smaller than $y$. For $y$-smooth numbers we will prove the following theorem. 
\begin{thm}
\label{lowerboundvariancesmooth2}
Let $N^{1/2+\delta}\leq Q\leq N$, with $\delta>0$ sufficiently small. Let $u:=(\log N)/(\log y)$. There exists a large constant $C>0$ such that the following holds. If 
$$1+\frac{\log C}{\log N}\leq u\leq 2$$
and $N$ is large enough in terms of $\delta$, we have
$$V(N,Q;\textbf{1}_{y-\text{smooth}})\gg_{\delta}QN\log u+Q^2.$$
\end{thm}
\begin{rmk}
We observe that Harper's result \cite[Theorem 2]{H} gives a tight corresponding upper bound for the variance above, when $Q=N/(\log N)^A$, with $A>0$ and $\sqrt{N}\leq y\leq N^{1-\delta}$, say.
\end{rmk} 
We will show our new theorems by first reducing ourselves to study certain $L^2$--integrals of the exponential sums with coefficients $1$, $\varpi(n)$, $d_{\alpha_N}^{\varpi}(n)$ or  $\textbf{1}_{y-\text{smooth}}(n)$, through an application of a technique introduced in a seminal work of Harper and Soundararajan. For them we will determine their size, which constitute an interesting result on its own and, since we believe that for the aforementioned functions their variance in arithmetic progressions should be well approximated by such integrals, will also give us a strong indication of the fact that our theorems should be sharp.
\section{Preliminary notions and results}
Throughout the rest of this paper the letter $p$ will be reserved for a prime number. Other letters might still indicate a prime number but in each case it will be specified.
\subsection{Some basic facts about certain arithmetic functions}
It is a classical result going back to Hardy and Ramanujan (see also Diaconis' paper \cite{D}) that the partial sum of the $\varpi$--function satisfies the following asymptotic expansion:
\begin{align}
\label{meanvarpi}
\sum_{n\leq x}\varpi(n)=x\log\log x+B_{\varpi}x+O\bigg(\frac{x}{\log x}\bigg)\ \ \ (x\geq 2),
\end{align}
where $B_{\varpi}$ is a constant depending on the function $\varpi$.
In particular, we deduce that the mean value of $\varpi(n)$, over the integers $n\leq x$, is roughly $\log\log x$. Regarding its variance, we can appeal to the Tur\'{a}n-Kubilius' inequality (see e.g. \cite[Ch. III, Theorem 3.1]{T}), which states that
\begin{align}
\label{variancevarpi}
\sum_{n\leq x}(\varpi(n)-\log\log n)^2\ll x\log\log x\ \ \ (x\geq 2).
\end{align}
In particular, \eqref{meanvarpi} and \eqref{variancevarpi} together give
\begin{align}
\label{secondmomentvarpi}
\sum_{n\leq x}\varpi(n)^2\ll x(\log\log x)^2\ \ \ (x\geq 2).
\end{align}
Finally, we remind of the following bound on the maximal size of $\varpi(n)$ (see e.g. \cite[Ch. I, Eq. 5.9]{T}):
\begin{align}
\label{maxsizevarpi}
\varpi(n)\leq (\log x)/(\log 2)\ \ \ (1\leq n\leq x).
\end{align}
We will make use of the following result on the partial sum of some non-negative multiplicative functions.
\begin{lem}
\label{Rankinestimate0}
For any non-negative multiplicative function $g(n)$ uniformly bounded on the prime numbers by a positive real constant $B$ and such that the sum $S=\sum_{q}(g(q)\log q)/q$ over all the prime powers $q=p^{k}$, with $k\geq 2$, converges, one has
$$\sum_{n\leq x}g(n)\ll_{B,S} \frac{x}{\log x}\sum_{n\leq x}\frac{g(n)}{n}\ \ \ (x\geq 2)$$
and 
\begin{equation*}
1\ll_{B,S} \sum_{n\leq x}\frac{g(n)}{n}\prod_{p\leq x}\bigg(1+\frac{g(p)}{p}\bigg)^{-1}\ll_{B,S} 1\ \ \ (x\geq 1).
\end{equation*}
\end{lem}
\begin{proof}
The first conclusion is \cite[Ch. III, Theorem 3.5]{T} and the second one is a special case of \cite[Lemma 20]{EK} of Elliott and Kish.
\end{proof}
In particular, we evidence the following immediate consequence for the partial sum of certain types of divisor functions (see e.g. \cite[Ch. III, Theorem 3.7]{T}).
\begin{cor}
\label{lemmapartialsumdiv}
Let $0<y_0<2$. Then, uniformly for $0\leq y\leq y_0$ and $x\geq 2$, one has
\begin{align*}
\sum_{n\leq x}y^{\varpi(n)}\ll x(\log x)^{y-1}.
\end{align*}
\end{cor}
\subsection{Preliminaries about the variance in arithmetic progressions}
As usual, we define the so called set of major arcs ${\frak M} = {\frak M}(K,Q_0,Q)$ consisting of those $\theta \in {\Bbb R}/{\Bbb Z}$ having an approximation $|\theta-a/q| \le K/(qQ)$ with moduli $q\le K Q_0$ and reduced residue classes $(a,q)=1$. Let instead ${\frak m}={\frak m}(K,Q_0,Q)$, the minor arcs, denote the complement of the major arcs in ${\Bbb R}/{\Bbb Z}$. Thus, the union of minor arcs is defined as the set of those real numbers in $[0,1]$ approximable by rational fractions with large denominator, as large as depending on $K,Q_0$ and $Q$.

As explained in \cite{M}, to produce lower bounds for the variance of complex sequences in arithmetic progressions we rely on an application of Harper and Soundararajan's method introduced in \cite{HS}, which points out a direct link between the variance and the $L^2$-norm of some exponential sums over unions of minor arcs. This is the content of  \cite[Proposition 1]{HS}, which we next report.
\begin{prop}
\label{Prop5.1} 
Let $f(n)$ be any complex sequence. Let $N$ be a large positive integer, $K\ge 5$ be a parameter and $K,Q_0$ and $Q$ be such that
\begin{equation} 
\label{eq0} 
K \sqrt{N\log N} \le Q \le N\ \textrm{and}\ \ \frac{N \log N}{Q} \le Q_0 \le \frac{Q}{K^2}. 
\end{equation} 
Then we have 
\begin{align}
\label{estimateprop1}
V(N,Q;f) &\ge Q\Big(1+ O\Big(\frac{\log K}{K}\Big)\Big) \int_{\frak m} |{\mathcal S}_{f}(\theta)|^2d\theta + O \Big( \frac{NK}{Q_0} \sum_{n\le N} |f(n)|^2 \Big)\\
&+O\bigg(\sum_{q\le Q } \frac{1}{q} \sum_{\substack {d|q \\ d>Q_0}} \frac{1}{\varphi(d)} \Big| \sum_{n\leq N} f(n) c_d(n)\Big|^2\bigg)\nonumber,
\end{align} 
where $c_d(n)$ are the Ramanujan sums defined as 
$$c_d(n)=\sum_{\substack{a=1,\dots,d\\ (a,d)=1}}e(an/d)$$
and ${\mathcal S}_{f}(\theta):=\sum_{n\leq N}f(n)e(n\theta)$ with $e(t)=e^{2\pi i t}$, for any $t\in\mathbb{R}$.
\end{prop} 
\begin{rmk}
One has the following representation for the Ramanujan sums as a sum over divisors (see e.g. \cite[Theorem 4.1]{MV}):
\begin{equation}
\label{mainpropc_q}
c_d(n)=\sum_{k|(n,d)}k\mu(d/k),
\end{equation} 
where $\mu(n)$ is the M\"{o}bius function.
\end{rmk}
To handle the $L^2$-integrals over minor arcs as in \eqref{estimateprop1} for the function $f(n)=\textbf{1}_{y-smooth}(n)$ we will appeal to \cite[Proposition 3]{HS}, which we next report adapted to our context.
\begin{prop}
\label{newprop10} 
Keep notations as above and assume $KQ_0< R:=N^{1/2-\delta/2}$. Then we have  
\begin{equation}
\label{C-S*}
\int_{\frak m} |{\mathcal S}_{f}(\theta)|^2d\theta\geq \bigg(\int_{\frak m} |{\mathcal S}_{f}(\theta)\mathcal{G}(\theta)|d\theta\bigg)^2\bigg(\int_{\frak m} |\mathcal{G}(\theta)|^2d\theta\bigg)^{-1},
\end{equation}
where
$$\mathcal{G}(\theta)=\sum_{n\leq N}\bigg(\sum_{\substack{r|n\\ r\leq R}}g(r)\bigg)e(n\theta),$$
for any complex arithmetic function $g(r)$. 

If moreover there exists a constant $\kappa>1$ for which $|g(n)|\leq d_{\kappa}(n)$, for any $n\leq N$, we also have
\begin{equation}
\label{eq81}
\int_{\frak m} |{\mathcal S}_{f}(\theta)\mathcal{G}(\theta) | d\theta \ge \sum_{KQ_0 < q\le R} \Big| \sum_{\substack{r\le R \\ q|r}} \frac{g(r)}{r} \Big| \Big| \sum_{\substack{n\le N}} 
f(n)c_q(n) \Big|+ O_{\delta,\kappa}(N^{1-\delta/11}), 
\end{equation}
if $N$ is large enough in terms of $\delta$ and $\kappa$.
\end{prop} 
\begin{proof}
This is a slight variation of \cite[Proposition 3]{HS} for functions bounded by a divisor function and the content of \cite[Lemma 5.2]{M}.
\end{proof}
To estimate the $L^2$--integrals over minor arcs as in \eqref{estimateprop1} for the functions $f(n)=d_{\alpha_N}^{\varpi}(n)$ and $f(n)=\varpi(n)$, we will instead need to invoke \cite[Proposition 2]{HS}, which we next report in a more compact form. For ease of readability, we say that a real smooth function $\phi(t)$ belongs to the ``Fourier class'' of functions $\mathcal{F}$ if: 
\begin{itemize}
\item $\phi(t)$ is compactly supported in $[0,1]$;
\item $0\leq \phi(t)\leq 1$, for all $0\leq t\leq 1$;
\item $\int_{0}^{1}\phi(t)dt\geq 1/2$;
\item $|\hat{\phi}(\xi)|\ll_{A}(1+|\xi|)^{-A}$, for any $A>0$, where $\hat{\phi}(\xi):=\int_{-\infty}^{+\infty}\phi(t)e(-\xi t)dt$ denotes the Fourier transform of $\phi(t)$.
\end{itemize}
\begin{prop}
\label{newprop11}
Keep notations as above and assume $KQ_0< R\leq Q/2K$. Then we have  
\begin{equation}
\label{C-S2}
\int_{\frak m} |{\mathcal S}_{f}(\theta)|^2d\theta\geq \bigg|\int_{\frak m} {\mathcal S}_{f}(\theta)\overline{\mathcal{G}(\theta)}d\theta\bigg|^2\bigg(\int_{\frak m} |\mathcal{G}(\theta)|^2d\theta\bigg)^{-1},
\end{equation}
where 
$$\mathcal{G}(\theta)=\sum_{n\leq N}\bigg(\sum_{\substack{r|n\\ r\leq R}}g(r)\bigg)\phi\left(\frac{n}{N}\right)e(n\theta),$$
for any complex arithmetic function $g(r)$ and real function $\phi(t)$. 

Let $M:=\max_{r\leq R}|g(r)|$. If $\phi(t)\in \mathcal{F}$, then we also have
\begin{align}
\label{lowerboundintminarcs2}
\int_{\frak m} {\mathcal S}_{f}(\theta)\overline{\mathcal{G}(\theta)}d\theta&=\sum_{n\leq N}f(n)\bigg(\sum_{\substack{r|n\\ r\le R}} \overline{g(r)}\bigg)\phi\bigg(\frac{n}{N}\bigg)\\
&-N\sum_{q\leq KQ_0}\int_{-K/qQ}^{K/qQ}\bigg(\sum_{n\leq N}f(n)c_q(n)e(n\beta)\bigg)\bigg(\sum_{\substack{r\leq R\\q|r}}\frac{\overline{g(r)}}{r}\bigg)\hat{\phi}(\beta N)d\beta\nonumber\\
&+O\bigg(\frac{MKR\sqrt{Q_0}\log N}{\sqrt{Q}}\sqrt{\sum_{n\leq N}f(n)^2}\bigg).\nonumber
\end{align}
\end{prop} 
It turns out that to lower bound the integral $\int_{\frak m} |{\mathcal S}_{\textbf{1}_{y-smooth}}(\theta)|^2d\theta$ is sufficient to only look at certain minor arcs, i.e. at those centred on fractions with denominators $q\leq R$. This makes the application of the Cauchy--Schwarz's inequality in the form \eqref{C-S*} efficient, which in turn simplifies our task by means of \eqref{eq81}. 
On the other hand, the contribution to the integrals $\int_{\frak{m}} |{\mathcal S}_{d_{\alpha_N}^{\varpi}}(\theta)|^2 d\theta$ and $\int_{\frak m} |{\mathcal S}_{\varpi}(\theta)|^2d\theta$ comes from all of the minor arcs, even from those centred on fractions with possibly very large denominators. This forces us to use the Cauchy--Schwarz's inequality as in \eqref{C-S2} and then to asymptotically estimate $\int_{\frak m} {\mathcal S}_{f}(\theta)\overline{\mathcal{G}(\theta)}d\theta$ by means of \eqref{lowerboundintminarcs2}.
\section{The $L^2$-integral of some exponential sums over minor arcs}
As already discussed, Harper and Soundararajan showed that, to lower bound the variance $V(N,Q;f)$ of complex arithmetic functions $f(n)$ in arithmetic progressions, we can switch our attention to integrals of exponential sums over unions of minor arcs, such as $\int_{\frak m} |{\mathcal S}_{f}(\theta)|^2d\theta$, for which we seek for a sharp lower bound. This is accomplished by an application of Proposition \ref{Prop5.1}. Our aim is to employ such strategy in the case $f(n)=1$, $f(n)=\varpi(n)$, $f(n)=d_{\alpha_N}^{\varpi}(n)$ and $f(n)=\textbf{1}_{y-\text{smooth}}(n)$ and with the choice of minor arcs $\frak{m}=\frak{m}(K,Q,Q_0)$ given by $K$ a large positive constant, $N^{1/2+\delta}\leq Q\leq N$, for any suitably small $\delta>0$, and $Q_0$ satisfying \eqref{eq0}. This will indeed be the underlying choice of minor arcs in the next propositions.

Regarding the constant function $1$, we have the following result.
\begin{prop}
\label{proplowerboundint1}
For any $N$ large enough with respect to $\delta$, we have
\begin{align}
\label{lowerboundint1}
\int_{\frak{m}} |{\mathcal S}_{1}(\theta)|^2 d\theta\gg Q.
\end{align}
\end{prop}
Regarding the additive function $\varpi(n)$, we will prove the next proposition.
\begin{prop}
\label{propL2integralofomega}
Suppose $KQ_0<N^{1/2-\delta/2}$. If $N$ is sufficiently large in terms of $\delta$, we have
\begin{equation}
\label{sizeintomega}
\int_{\frak m} |{\mathcal S}_{\varpi}(\theta)|^2 d\theta\gg_\delta Q(\log\log N)^2+N\log\bigg(\frac{\log N}{\log (2N/Q)}\bigg).
\end{equation}
\end{prop}
Regarding the multiplicative function $d_{\alpha_N}^{\varpi}(n)$, the result is the following.
\begin{prop}
\label{proplowerboundintomega}
Suppose $KQ_0<N^{1/2-\delta/2}$. There exists a large constant $C=C(\delta)>0$ such that if $C\log\log N< |R(N)|\leq N^{\delta/12}$ and $N$ is large enough in terms of $\delta$, we have
\begin{equation}
\label{lowerboundintomega1.5}
\int_{\frak{m}} |{\mathcal S}_{d_{\alpha_N}^{\varpi}}(\theta)|^2 d\theta\gg_\delta \frac{N}{R(N)^2}\log\bigg(\frac{\log N}{\log(2N/Q)}\bigg)+Q.
\end{equation}
\end{prop}
\begin{rmk}
From the proof of \cite[Theorem 1.11]{M} it can be easily evinced that
\begin{align*}
\int_{\frak{m}} |{\mathcal S}_{d_{\alpha_N}^{\varpi}}(\theta)|^2 d\theta&\gg_\delta \frac{N}{R(N)^4}\exp\bigg(\bigg(2+\frac1{R(N)}\bigg)\frac{\log\log N}{R(N)}\bigg),
\end{align*}
whenever $B<|R(N)|\leq \log\log N$, for a suitable large constant $B\geq 3$. 
\end{rmk}
Regarding the indicator of $y$--smooth numbers, we will show the following lower bound.
\begin{prop}
\label{proplowerboundvariancesmooth2}
Assume that $KQ_0\leq N^{1/2-\delta}(\log N)^{17}$. Let $u:=(\log N)/(\log y)$. There exists a large constant $C>0$ such that the following holds. If 
$$1+\frac{\log C}{\log N}\leq u\leq 2$$
and $N$ is large enough in terms of $\delta$, we have
\begin{align}
\label{lowerboundsmoothint}
\int_{\frak m} |{\mathcal S}_{\textbf{1}_{y-\text{smooth}}}(\theta)|^2 d\theta\gg_{\delta} N\log u+Q.
\end{align}
\end{prop}
In order to show that $Q$-times our lower bounds \eqref{lowerboundint1},  \eqref{sizeintomega}, \eqref{lowerboundintomega1.5} and \eqref{lowerboundsmoothint} provides us with the expected best possible approximation for the related variances, we will produce corresponding sharp upper bounds for them, which in some cases will also turn out to be useful to deduce the aforementioned lower bounds themselves. 
\begin{prop}
\label{propupperbounds}
With notations as in Propositions \ref{proplowerboundint1}, \ref{propL2integralofomega}, \ref{proplowerboundintomega} and \ref{proplowerboundvariancesmooth2}, we have that
\begin{enumerate}[a)]
\item \eqref{lowerboundint1} is sharp;
\item \eqref{sizeintomega} is sharp; 
\item the estimate \eqref{lowerboundintomega1.5} is sharp when $|R(N)|>(\log\log N)^{3/2}$;
\item \eqref{lowerboundsmoothint} is sharp.
\end{enumerate}
\end{prop}
\begin{rmk}
\label{rmkonupperboundsdiv}
It should be possible to produce a sharp upper bound for the integral in \eqref{lowerboundintomega1.5} in the whole range $|R(N)|>C\log\log N$ (see Remark \ref{RemarksmallerR} below).
\end{rmk}
To work out the size of the $L^2$-integral over minor arcs of the exponential sum with coefficients $f(n)=d_{\alpha_N}^{\varpi}(n)$, we will split $f$ into a sum $f=f_d+f_r$ of a deterministic part $f_d$, constant, and a pseudorandom one $f_r$. By triangle inequality we will separate their contribution to the integrals to then analyse them individually. To deal with $\int_{\frak m} |{\mathcal S}_{f_d}(\theta)|^2 d\theta$ we will unfold the definition of minor arcs and insert classical estimates for the size of a complete exponential sum. Regarding $\int_{\frak m} |{\mathcal S}_{f_r}(\theta)|^2 d\theta$ instead, when $|R(N)|> (\log\log N)^{3/2}$, we will reduce the problem to estimate the $L^2$-integral over minor arcs of the exponential sum with coefficients $\varpi(n)$. To this aim, we will write $\varpi(n)=\Sigma_1+\Sigma_2$, where $\Sigma_1$ is a sum over prime numbers smaller than a power of $2N/Q$ and $\Sigma_2$ the remaining part, and again use triangle inequality. To estimate $\int_{\frak m} |{\mathcal S}_{\Sigma_2}(\theta)|^2 d\theta$ we will use Parseval's identity and an application of Tur\'{a}n--Kubilius' inequality. Regarding $\int_{\frak m} |{\mathcal S}_{\Sigma_1}(\theta)|^2 d\theta$ instead we will expand out the square inside the integral and unfold the definition of minor arcs to then conclude by counting the number of primes which are solution to certain systems of congruences.
\section{Proof of Proposition \ref{propupperbounds}}
We set the parameter $K$ to be a large constant, $N^{1/2+\delta}\leq Q\leq N$, with $N$ sufficiently large in terms of $\delta$, and $Q_0$ satisfying \eqref{eq0}. We keep these notations throughout the rest of this section. 
\subsection{The case of the constant function $1$}
We use the well-known bound
\begin{align}
\label{fundamentalestimate}
|{\mathcal S}_{1}(\theta)|\ll \min\bigg\{N,\frac1{||\theta||}\bigg\},
\end{align}
where $||\theta||$ indicates the distance of $\theta$ from the nearest integer. 
Since $\theta=a/q+\delta$, with $|\delta|\leq K/qQ$ and $q>KQ_0$, we have that either $||\theta||=|\theta|$ or $||\theta||=1-|\theta|$. 
Hence, by symmetry, we find that 
\begin{align*}
\int_{\frak{m}} |{\mathcal S}_{1}(\theta)|^2 d\theta\ll \sum_{KQ_0<q\leq Q}\sum_{\substack{1\leq a< q/2\\ (a,q)=1}}\int_{a/q-\frac{K}{qQ}}^{a/q+\frac{K}{qQ}}\frac{1}{\theta^2}d\theta&=\frac{2K}{Q}\sum_{KQ_0<q\leq Q}q\sum_{\substack{1\leq a<q/2\\ (a,q)=1}}\frac{1}{a^2-(K/Q)^2}\ll Q,
\end{align*}
where we used that $a^2-(K/Q)^2\geq a^2/2$, for any $a\geq 1$, if $N$ is large enough. This shows Proposition \ref{propupperbounds} a).
\subsection{The case of smooth numbers}
We first observe that for any two complex numbers $w,z$ we have 
\begin{align}
\label{trivialineq}
|w+z|^2\leq 2(|w|^2+|z|^2).
\end{align}
By writing $\textbf{1}_{y-\text{smooth}}(n)=1-\textbf{1}_{\exists p|n: p>y}(n)$ and using \eqref{trivialineq} to separate their contribution to the integral, we get
\begin{align*}
\int_{\frak m} |{\mathcal S}_{\textbf{1}_{y-\text{smooth}}}(\theta)|^2 d\theta\ll Q+\sum_{\substack{n\leq N\\ \exists p|n: p>y}}1\leq Q+N\sum_{y<p\leq N}\frac{1}{p}\ll Q+ N\log u,
\end{align*}
by Proposition \ref{propupperbounds} a), Parseval's identity and Mertens' theorem, where $u:=(\log N)/(\log y)\in [1+1/\log N, 2].$ This shows Proposition \ref{propupperbounds} d).
\subsection{The case of divisor functions close to $1$}
Let $\alpha_N=1+1/R(N)$, where $R(N)$ is a non-vanishing real function with $|R(N)|>C\log\log N$, for a constant $C>0$ to determine later on.
By \eqref{trivialineq}, one has
\begin{equation}
\label{triangleineq}
\int_{\mathfrak{m}}|{\mathcal S}_{d_{\alpha_N}^{\varpi}}(\theta)|^2d\theta\leq 2\int_{\mathfrak{m}}|{\mathcal S}_{d_{\alpha_N}^{\varpi}-1}(\theta)|^2d\theta+2\int_{\mathfrak{m}}|{\mathcal S}_{1}(\theta)|^2d\theta
\end{equation}
and we split the exponential sum with coefficients $d_{\alpha_N}^{\varpi}(n)-1$ according to whether $\varpi(n)\leq A\log\log N$ or $\varpi(n)>A\log\log N$, with $A>0$ large to be chosen later.  We do this only when $|R(N)|\leq (\log N)/(\log 2)$. We separate their contribution to the integral by \eqref{trivialineq}. The second one is bounded by Parseval's identity by
\begin{align*}
\sum_{\substack{n\leq N\\ \varpi(n)> A\log\log N}}(\alpha_N ^{\varpi(n)}-1)^2&\leq \sum_{\substack{n\leq N\\ \varpi(n)> A\log\log N}}(\alpha_N^{2\varpi(n)}+1)\\
&\leq \frac{1}{(\log N)^{A\log(5/4)}}\sum_{\substack{n\leq N}}\bigg(\bigg(\frac{3}{2}\bigg)^{\varpi(n)}+1\bigg)\bigg(\frac{5}{4}\bigg)^{\varpi(n)}\ll \frac{N}{(\log N)^{3}},
\end{align*}
say, by Corollary \ref{lemmapartialsumdiv} and choosing $A$ large enough.

Let 
\begin{equation}
\label{defErr}
\text{Err}(N):= \left\{ \begin{array}{ll}
         \frac{N}{(\log N)^{3}}& \mbox{if $|R(N)|\leq (\log N)/(\log 2)$};\\
         0 & \mbox{otherwise}.\end{array} \right.
\end{equation}
From the above considerations and Proposition \ref{propupperbounds} a), we deduce that
\begin{equation*}
\int_{\mathfrak{m}}|{\mathcal S}_{d_{\alpha_N}^{\varpi}}(\theta)|^2d\theta\ll\int_{\mathfrak{m}}\bigg|\sum_{\substack{n\leq N\\ \varpi(n)\leq A\log\log N}}(\alpha_N^{\varpi(n)}-1)e(n\theta)\bigg|^2d\theta+Q+\text{Err}(N),
\end{equation*}
where the restriction on the sum is there only when  $|R(N)|\leq (\log N)/(\log 2)$. The integral on the right-hand side by \eqref{trivialineq} is
\begin{align*}
\ll \frac{1}{R(N)^2}\int_{\mathfrak{m}}\bigg|\sum_{\substack{n\leq N\\ \varpi(n)\leq A\log\log N}}\varpi(n)e(n\theta)\bigg|^2d\theta+\int_{\mathfrak{m}}|{\mathcal S}_{T_N}(\theta)|^2d\theta,
\end{align*}
where we let
$$T_N(n):=\bigg(\alpha_N^{\varpi(n)}-1-\frac{\varpi(n)}{R(N)}\bigg)\textbf{1}_{\varpi(n)\leq A\log\log N}.$$
The second integral above, again by \eqref{trivialineq}, is
\begin{align*}
\ll M_N^2\int_{\mathfrak{m}}|{\mathcal S}_{1}(\theta)|^2d\theta+\int_{\mathfrak{m}}|{\mathcal S}_{T_N-M_N}(\theta)|^2d\theta,
\end{align*}
where 
$$M_N:=\alpha_N^{\log\log N}-1-\frac{\log\log N}{R(N)}.$$
By Proposition \ref{propupperbounds} a), the first term above is $\ll Q(\log\log N)^4/R(N)^4\leq Q,$ if $C$ is large enough. On the other hand, the second one, by Parseval's identity, by Taylor expanding $\alpha_N^{\varpi(n)}$ and $\alpha_N^{\log\log N}$ (which we can do thanks to the restriction in the sum and reminding of the maximal size \eqref{maxsizevarpi} of $\varpi(n)$) and using the well-known identity $a^k-b^k=(a-b)\sum_{j=0}^{k-1}a^jb^{k-1-j}$, which holds for a couple of positive real numbers $a,b$ and any positive integer $k$, can be estimated with
\begin{align*}
\ll \frac{(\log\log N)^2}{R(N)^4}\sum_{\substack{n\leq N}}(\varpi(n)\textbf{1}_{\varpi(n)\leq A\log\log N}-\log\log N)^2\ll \frac{N(\log\log N)^3}{R(N)^4},
\end{align*}
if $A,C(A)$ and $N$ are sufficiently large, by inserting and after removing the condition $\varpi(n)\leq A\log\log N$ on the sum, at a cost of an acceptable error term, and performing the mean square estimate using \eqref{variancevarpi}.
Overall, by gathering all of the above considerations, we have showed that
\begin{equation}
\label{boundexpsumdiv}
\int_{\frak{m}} |{\mathcal S}_{d_{\alpha_N}^{\varpi}}(\theta)|^2 d\theta \ll \frac{1}{R(N)^2}\int_{\mathfrak{m}}|{\mathcal S}_{\varpi}(\theta)|^2d\theta+Q+\frac{N(\log\log N)^3}{R(N)^4}+\frac{N}{R(N)^2\log N},
\end{equation}
say, whenever $|R(N)|>C\log\log N$ and $C$ and $N$ are sufficiently large.

It is then clear that assuming the upper bound in Proposition \ref{propupperbounds} b) for $\int_{\frak m} |{\mathcal S}_{\varpi}(\theta)|^2 d\theta$ and $|R(N)|>(\log\log N)^{3/2}$
we get Proposition \ref{propupperbounds} c). 
\begin{rmk}
\label{RemarksmallerR}
If we had $T_N(n)=\varpi(n)^2/(2R(N)^2)$, we believe that we would roughly find
\begin{align*}
\int_{\mathfrak{m}}|{\mathcal S}_{T_N}(\theta)|^2d\theta\approx \frac{N(\log\log N)^2}{R(N)^4}\log\bigg(\frac{\log N}{\log(2N/Q)}\bigg).
\end{align*}
This would imply that the lower bound \eqref{lowerboundintomega1.5} for the integral $\int_{\frak{m}} |{\mathcal S}_{d_{\alpha_N}^{\varpi}}(\theta)|^2 d\theta$ is sharp in the whole range $|R(N)|>C\log\log N$, with $C$ large.
In practice, by writing $T_N(n)$ as a truncated Taylor series up to order $k$, plus a remainder term, we believe we would get to prove that \eqref{lowerboundintomega1.5} is sharp in the range $|R(N)|>(\log\log N)^{1+1/k}$, for any fixed positive integer $k$, by inspecting the structure of the minor arcs. 
Even though this would constitute an improvement over the result of Proposition \ref{propupperbounds} c), we will not commit ourselves to formally proving this here.
\end{rmk}
\subsection{The case of the $\varpi$ function}
To begin with, we write
\begin{equation*}
\sum_{n\leq N}\omega(n)e(n\theta)=\sum_{n\leq N}\omega_1(n)e(n\theta)+\sum_{n\leq N}\omega_2(n)e(n\theta),
\end{equation*}
where $\omega_1(n)$ is the number of prime factors of $n$ smaller than $\sqrt[4]{2N/Q}$ and $\omega_2(n)$ that of prime divisors contained in the interval $(\sqrt[4]{2N/Q}, N]$.
By \eqref{trivialineq}, one has
\begin{align}
\label{removingsmoothdependence}
\int_{\mathfrak{m}}|S_{\omega}(\theta)|^2d\theta\ll\int_{\mathfrak{m}}|S_{\omega_1}(\theta)|^2d\theta+\int_{\mathfrak{m}}|S_{\omega_2}(\theta)|^2d\theta.
\end{align}
A simple calculation shows that $\omega_2(n)$ has a mean value of size $\log((4\log N)/(\log(2N/Q))).$
Hence, isolating this term inside the corresponding integral gives
\begin{align*}
\label{secondintomega}
\int_{\mathfrak{m}}|S_{\omega_2}(\theta)|^2d\theta&\ll Q\bigg(\log\bigg(\frac{4\log N}{\log (2N/Q)}\bigg)\bigg)^2+ \int_{\frak m} \bigg|\sum_{\substack{n\leq N}}\bigg(\omega_2(n)-\log\bigg(\frac{4\log N}{\log (2N/Q)}\bigg)\bigg)e(n\theta)\bigg|^2 d\theta\\
&\ll Q\bigg(\log\bigg(\frac{4\log N}{\log (2N/Q)}\bigg)\bigg)^2+N\log\bigg(\frac{4\log N}{\log (2N/Q)}\bigg),\nonumber
\end{align*}
by Proposition \ref{propupperbounds} a), Parseval's identity and an application of the general form of the Tur\'{a}n--Kubilius' inequality, which gives an analogue for $\omega_2(n)$ of \eqref{variancevarpi} (see e.g. \cite[Ch. III, Theorem 3.1]{T}). 
Moreover, from 
\begin{align*}
\sum_{n\leq N}\Omega(n)e(n\theta)=\sum_{n\leq N}\omega(n)e(n\theta)+\sum_{n\leq N}\bigg(\sum_{\substack{p^k|n\\ k\geq 2}}1\bigg)e(n\theta)
\end{align*}
we immediately get
\begin{align*}
\int_{\mathfrak{m}}|{\mathcal S}_{\Omega}(\theta)|^2d\theta\ll \int_{\mathfrak{m}}|{\mathcal S}_{\omega}(\theta)|^2d\theta+\sum_{n\leq N}\bigg(\sum_{\substack{p^k|n\\ k\geq 2}}1\bigg)^2
\end{align*}
and, by expanding the square out and swapping summations, we see that the above sum is
\begin{align*}
\sum_{n\leq N}\sum_{\substack{p_1^k|n\\ k\geq 2}}\sum_{\substack{p_2^j|n\\ j\geq 2}}1&=\sum_{p_1\leq \sqrt{N}}\sum_{k=2}^{\left\lfloor \frac{\log N}{\log p_1}\right\rfloor}\sum_{p_2\leq \sqrt{N}}\sum_{j=2}^{\left\lfloor \frac{\log N}{\log p_2}\right\rfloor}\sum_{\substack{n\leq N\\ n\equiv 0\pmod{[p_1^k,p_2^j]}}}1\\
&\leq N\sum_{p_1\leq \sqrt{N}}\sum_{k=2}^{\left\lfloor \frac{\log N}{\log p_1}\right\rfloor}\sum_{j=2}^{\left\lfloor \frac{\log N}{\log p_1}\right\rfloor}\frac{1}{p_1^{\max\{k,j\}}}+N\sum_{p_1\leq \sqrt{N}}\sum_{k=2}^{\left\lfloor \frac{\log N}{\log p_1}\right\rfloor}\frac{1}{p_1^k}\sum_{\substack{p_2\leq \sqrt{N}\\ p_2\neq p_1}}\sum_{j=2}^{\left\lfloor \frac{\log N}{\log p_2}\right\rfloor}\frac{1}{p_2^j}\ll N.
\end{align*}
For the rest of this section, we will focus on showing the following statement.
\begin{claim}
\label{Claim2}
Let $K$ be a large constant, $N^{1/2+\delta}\leq Q\leq N$, with $N$ sufficiently large in terms of $\delta$, and $Q_0$ satisfying \eqref{eq0}. Then we have
\begin{align*}
\int_{\frak m} |{\mathcal S}_{\omega_1}(\theta)|^2 d\theta\ll N.
\end{align*}
\end{claim}
Assuming the validity of Claim \ref{Claim2}, and collecting the above observations together, it is immediate to deduce Proposition \ref{propupperbounds} b). We now then move to the proof of Claim \ref{Claim2}. By expanding the integral, we find
\begin{equation}
\label{expansionint}
\int_{\frak m} |{\mathcal S}_{\omega_1}(\theta)|^2 d\theta=\sum_{KQ_0<q\leq Q}\sum_{\substack{a=1,\dots,q\\ (a,q)=1}}\int_{a/q-K/qQ}^{a/q+K/qQ}\bigg|\sum_{p\leq \sqrt[4]{2N/Q}}\sum_{\substack{k\leq N/p}}e(kp\theta)\bigg|^2d\theta.
\end{equation}
We observe that each innermost exponential sum is quite ``long'', since for any $p\leq \sqrt[4]{2N/Q}$, it always runs over at least $Q$ numbers. We thus expect to individually observe cancellation. Hence, we should not lose much by trivially upper bounding the double sum using the triangle inequality followed by \eqref{fundamentalestimate}.
Since $p\theta=pa/q+p\beta$ and by \eqref{eq0} 
$$|p\beta|\ll \frac{N}{qQ^2}\leq \frac1{q}\leq \frac1{Q_0}\leq \frac1{\log N},$$ 
we deduce that
$$||p\theta||=||\overline{pa}/q+p\beta||=\min\{|\overline{pa}/q+p\beta|,|1-\overline{pa}/q-p\beta|\},$$
where $\overline{pa}$ stands for the residue class of $pa$ modulo $q$. We will only focus on the case $\overline{pa}\leq q/2$, so that the above minimum always coincides with $|\overline{pa}/q+p\beta|$, since the complementary one can be similarly dealt with. We notice that $\overline{pa}>0$. For, if $\overline{pa}=0$ then $q|p$ and $p\leq 2N/Q$, which cannot happen since $q>KQ_0$. Hence, $|\overline{pa}/q+p\beta|\geq \overline{pa}/2q$. Indeed, for any $N$ large enough compared to $\delta$, we have
$$p|\beta|\leq\frac{KN}{qQ^2}\leq \frac1{2q}\leq \frac{\overline{pa}}{2q}.$$
Putting together the above information, we see that \eqref{expansionint} is
\begin{align}
\label{mainstartingpoint}
&\ll \frac{1}{Q}\sum_{KQ_0<q\leq Q}\frac{1}{q}\sum_{\substack{a=1,\dots,q\\ (a,q)=1}}\bigg(\sum_{\substack{p\leq \sqrt[4]{2N/Q}}}\min\bigg\{\frac{N}{p},\frac{q}{\overline{pa}}\bigg\}\bigg)^2.
\end{align}
Note that the above minimum is always of size $q/\overline{pa}$. So, the above reduces to be  
\begin{align}
\label{mainstartingpoint2}
&=\frac{1}{Q}\sum_{KQ_0<q\leq Q}q\sum_{\substack{a=1,\dots,q\\ (a,q)=1}}\bigg(\sum_{\substack{p\leq \sqrt[4]{2N/Q}}}\frac1{\overline{pa}}\bigg)^2\\
&=\frac{1}{Q}\sum_{KQ_0<q\leq Q}q\sum_{\substack{a=1,\dots,q\\ (a,q)=1}}\sum_{\substack{p_1,p_2\leq \sqrt[4]{2N/Q}}}\frac{1}{\overline{p_1a}}\frac{1}{\overline{p_2a}}\nonumber\\
&\leq \frac{1}{Q}\sum_{KQ_0<q\leq Q}q\sum_{\substack{p_1,p_2\leq \sqrt[4]{2N/Q}}}\sum_{\substack{b_1, b_2\leq q}}\frac{1}{b_1b_2}\sum_{\substack{a=1,\dots,q\\ (a,q)=1\\ p_1a\equiv b_1\pmod q\\p_2a\equiv b_2\pmod q}}1.\nonumber
\end{align}
The system of congruences 
\begin{equation*}
\left\{ \begin{array}{ll}
         p_1a\equiv b_1\pmod q\\
         p_2a\equiv b_2\pmod q\end{array} \right.
\end{equation*}
has always at most $\min\{p_1,p_2\}$ solutions. By multiplying through the first equation by $b_2$ and the second one by $b_1$ we need to have
$$p_1b_2a\equiv p_2b_1a \pmod q \Leftrightarrow p_1b_2\equiv p_2b_1 \pmod q. $$
Therefore, we may upper bound the quantity in the last line of \eqref{mainstartingpoint2} with
\begin{align}
\label{secondcasecongruence}
\frac{1}{Q}\sum_{KQ_0<q\leq Q}q\sum_{\substack{p_1,p_2\leq \sqrt[4]{2N/Q}}}\min\{p_1,p_2\}\sum_{\substack{b_1,b_2\leq q\\ p_1b_2\equiv p_2b_1 \pmod q}}\frac{1}{b_1b_2}.
\end{align}
It is easy to verify that we have at most $p_1$ solutions $b_2\pmod{q}$ of the congruence relation $p_1b_2\equiv p_2b_1 \pmod q$, with $b_2\geq p_2b_1/p_1$. Hence \eqref{secondcasecongruence} may be upper bounded by
\begin{align*}
\frac{1}{Q}\sum_{KQ_0<q\leq Q}q\sum_{\substack{p_1,p_2\leq \sqrt[4]{2N/Q}}}\frac{p_1^2 \min\{p_1,p_2\}}{p_2}\sum_{\substack{b_1\leq q}}\frac{1}{b_1^2}\ll \frac{1}{Q}\sum_{KQ_0<q\leq Q}q\sum_{\substack{p_1,p_2\leq \sqrt[4]{2N/Q}}}p_1^2\ll N,
\end{align*}
thus concluding the proof of Claim \ref{Claim2}.
\begin{rmk}
Note that we have been able to facilitate the estimate of \eqref{mainstartingpoint2} thanks to our choice of parameter $\sqrt[4]{2N/Q}$ in \eqref{removingsmoothdependence}.
\end{rmk}
\section{The partial sum of some arithmetic functions twisted with Ramanujan sums}
A key step to find a lower bound for the variance of a function $f$ in arithmetic progressions is to produce a lower bound for the $L^2$-integral over minor arcs of the exponential sum with coefficients $f(n)$. For smooth numbers, this will be accomplished by means of Proposition \ref{newprop10}. More specifically, \eqref{eq81} allows us to reduce the problem to asymptotically estimate the partial sum of $f(n)=\textbf{1}_{y-\text{smooth}}(n)$ twisted with the Ramanujan sums $c_q(n)$. This will indeed constitute a crucial point in our argument and next we are going to state and prove the relative result.
\begin{lem}
\label{lemsumsmoothtwisted}
Let $C$ be a sufficiently large positive constant and consider $\sqrt{N}\leq y\leq N/C$. Then for any prime number $\log N<q\leq \sqrt{N}$ and $N$ large enough, we have
\begin{align*}
\bigg|\sum_{\substack{n\leq N\\ p|n\Rightarrow p\leq y}}c_q(n)\bigg|\gg N\log\bigg(\frac{\log N}{\log(\max\{N/q,y\})}\bigg)
\end{align*}
and for any squarefree positive integer $1<q\leq \sqrt{N}$ with all the prime factors larger than $N/y$, we have
\begin{align*}
\bigg|\sum_{\substack{n\leq N\\ p|n\Rightarrow p\leq y}}c_q(n)\bigg|\gg N\log u,
\end{align*}
where $u:=(\log N)/(\log y)$.
\end{lem}
\begin{proof}
By \cite[Ch. III, Theorem 5.8]{T} we know that
\[\Psi\left(\frac{N}{d},y\right):=\sum_{\substack{n\leq N/d\\ p|n\Rightarrow p\leq y}}1= \left\{ \begin{array}{ll}
        \lfloor\frac{N}{d}\rfloor & \mbox{if $d> N/y$};\\
     \frac{N}{d}(1-\log(\frac{\log(N/d)}{\log y}))+O(\frac{N}{d\log y}) & \mbox{if $d\leq N/y$}.\end{array} \right. \]
For any prime number $q$ the identity \eqref{mainpropc_q} reduces to $c_q(n)=-1+q\textbf{1}_{q|n}.$ It is then immediate to verify the following equality:
\begin{align*}
\sum_{\substack{n\leq N\\ p|n\Rightarrow p\leq y}}c_q(n)=-\Psi(N,y)+q\Psi\left(\frac{N}{q},y\right),
\end{align*}
from which it is straightforward to deduce the first estimate of the lemma.

By \eqref{mainpropc_q}, and letting $\sigma(q):=\sum_{d|q}d$, we can always rewrite the sum in the statement as
\begin{align*}
\sum_{d|q}d\mu\left(\frac{q}{d}\right)\Psi\left(\frac{N}{d},y\right)&=N\sum_{\substack{d|q\\d> N/y}}\mu\left(\frac{q}{d}\right)+N\sum_{\substack{d|q\\d\leq N/y}}\mu\left(\frac{q}{d}\right)\bigg(1-\log\bigg(\frac{\log(N/d)}{\log y}\bigg)\bigg)\nonumber\\
&+ O\bigg(\frac{N}{\log N}\sum_{\substack{d|q\\d\leq N/y}}1+\sigma(q)\bigg).\nonumber
\end{align*}
In the hypothesis that $q>1$ has all the prime factors larger than $N/y$, the sums over the divisors of $q$ smaller than or equal to $N/y$ reduce only to the single term corresponding to $d=1$. Hence, we actually have
\begin{align*}
\sum_{\substack{n\leq N\\ p|n\Rightarrow p\leq y}}c_q(n)=-N\mu(q)+N\mu(q)(1-\log u)+O\bigg(\frac{N}{\log N}\bigg)=-N\mu(q)\log u+O\bigg(\frac{N}{\log N}\bigg),
\end{align*}
since $\sigma(q)\ll q\log\log q\ll \sqrt{N}\log\log N\leq N/\log N$ (see \cite[Ch. I, Theorem 5.7]{T}), if $N$ is large, which immediately leads to deduce the second estimate of the lemma.
\end{proof}
To prove the lower bound for the variance of $\varpi(n)$ and of $d_{\alpha_N}^{\varpi}(n)$ in arithmetic progressions we will instead invoke Proposition \ref{newprop11}. To this aim, we need to study the partial sum of $\varpi(n)$ twisted with the Ramanujan sums and weighted by the smooth weight $\phi(n/N)$, with $\phi(t)$ belonging to the Fourier class $\mathcal{F}$ as in Proposition \ref{newprop11}.
\begin{lem}
\label{lempartialsumsmoothweight}
Let $R:=N^{1/2-\delta/2}$, for $\delta>0$ small, and suppose that $N^{1/2+\delta}\leq Q\leq cN/\log\log N$, for a certain absolute constant $c>0$.
Then  for any $N$ large enough with respect to $\delta$, we have
\begin{align*}
\bigg|\sum_{\substack{2N/Q<p\leq R}}\frac{1}{p}\sum_{n\leq N}\varpi(n)c_p(n)\phi\left(\frac{n}{N}\right)\bigg|\gg N\log\bigg(\frac{\log R}{\log(2N/Q)}\bigg).
\end{align*}
\end{lem}
\begin{proof}
To begin with, we note that for prime numbers $p$ the identity \eqref{mainpropc_q} reduces to $c_p(n)=-1+p\textbf{1}_{p|n}.$ Hence, the sum over $n$ in the statement is
\begin{align*}
&=-\sum_{n\leq N}\varpi(n)\phi\left(\frac{n}{N}\right)+p\sum_{\substack{n\leq N\\ p|n}}\varpi(n)\phi\left(\frac{n}{N}\right)\\
&=-\sum_{n\leq N}\varpi(n)\phi\left(\frac{n}{N}\right)+p\sum_{\substack{k\leq N/p}}(\varpi(k)+1)\phi\left(\frac{kp}{N}\right)+O\bigg(p\sum_{\substack{k\leq N/p^2}}(\varpi(k)+2)\bigg),
\end{align*}
where we used that $\varpi(pk)\leq \varpi(k)+\varpi(p)=\varpi(k)+1$. By \eqref{meanvarpi} the above big-Oh error term contributes at most $\ll N(\log\log N)/p$.

By partial summation from \eqref{meanvarpi}, it is easy to show that
\begin{align*}
\sum_{n\leq N}\varpi(n)\phi\left(\frac{n}{N}\right)=JN\log\log N+JNB_\varpi+O\bigg(\frac{N\log\log N}{\log N}\bigg),
\end{align*}
for any $N$ large enough, where $J:=\int_{0}^{1}\phi(t)dt\in [1/2, 1].$ This, applied once with $N$ and once with $N/p$, together with the previous observations, gives 
\begin{align*}
\sum_{n\leq N}\varpi(n)c_p(n)\phi\left(\frac{n}{N}\right)= JN\bigg(1+\log\bigg(1-\frac{\log p}{\log N}\bigg)\bigg)+O\bigg(\frac{N\log\log N}{\log N}+\frac{N\log\log N}{p}\bigg).
\end{align*}
Therefore, we see that the double sum in the statement is
\begin{align*}
&=JN\sum_{\substack{2N/Q<p\leq R}}\frac{1+\log(1-\frac{\log p}{\log N})}{p}+O\bigg(\frac{N(\log\log N)^2}{\log N}+N\log\log N\sum_{\substack{2N/Q<p\leq R}}\frac{1}{p^2}\bigg)\\
&\gg N\log\bigg(\frac{\log R}{\log(2N/Q)}\bigg)+O\bigg(\frac{N(\log\log N)^2}{\log N}+\frac{Q\log\log N}{\log(2N/Q)}\bigg),
\end{align*}
by Mertens' theorem, from which the thesis follows on our range of $Q$, if $N$ is large enough with respect to $\delta$.
\end{proof}
The next result shows a huge amount of cancellation for the partial sum of a Ramanujan sum weighted with an exponential phase.
\begin{lem}
\label{lemramanujansmooth}
Let $R:=N^{1/2-\delta/2}$, for $\delta>0$ small, and $q<R$ be a prime number. Then we have
\begin{align*}
&\sum_{n\leq N}c_q(n)e\left(\frac{nu}{N}\right)\ll q(1+|u|),
\end{align*}
uniformly for all real numbers $u$.
\end{lem}
\begin{proof}
To begin with, we notice that for any prime number $q$, the following estimate holds:
\begin{align}
\label{defandpropS}
S(t):=\sum_{\substack{n\leq t}}c_q(n)=\sum_{\substack{n\leq t\\ q|n}}q-\sum_{n\leq t}1\ll q,
\end{align}
by \eqref{mainpropc_q}, for any $t\geq 1$. Hence, by partial summation we find
\begin{align*}
\sum_{n\leq N}c_q(n)e\left(\frac{nu}{N}\right)=\int_{1}^{N}e\left(\frac{tu}{N}\right)dS(t)&=S(N)e(u)-S(1)e\left(\frac{u}{N}\right)-\frac{u}{N}\int_{1}^{N}S(t)e\left(\frac{tu}{N}\right)dt,
\end{align*}
from which, by using \eqref{defandpropS}, the thesis follows.
\end{proof}
The last result of this section, preliminary to the proof of the lower bound for $\int_{\frak{m}} |{\mathcal S}_{d_{\alpha_N}^{\varpi}}(\theta)|^2 d\theta$ contained in Proposition \ref{proplowerboundintomega}, concerns the partial sum of the divisor functions $d_{\alpha_N}^{\varpi}(n)$ twisted with Ramanujan sums and weighted by $\phi(n/N)$, with $\phi(t)$ belonging to the Fourier class $\mathcal{F}$ as in Proposition \ref{newprop11}.
\begin{lem}
\label{lempartialsumdivtwisted}
Let $\alpha_N=1+1/R(N)$, where $R(N)$ is a non-zero real function, and $R:=N^{1/2-\delta/2}$, for $\delta>0$ small. Assume $N^{1/2+\delta}\leq Q<cN(\log\log N)/R(N)^2$, for a certain absolute constant $c>0$. There exists a sufficiently large constant $C=C(\delta)>0$ such that if $C\log\log N\leq |R(N)|\leq (\log\log N)^{3}$ and $N$ is large enough with respect to $\delta$, we have
\begin{align*}
\bigg|\sum_{2N/Q<p\leq R}\sum_{\substack{n\leq N}}d_{\alpha_N}^{\varpi}(n)c_p(n)\phi\left(\frac{n}{N}\right)\bigg|\gg \frac{N}{|R(N)|}\log\bigg(\frac{\log R}{\log(2N/Q)}\bigg).
\end{align*}
\end{lem}
\begin{proof}
By adapting the proof of \cite[Theorem 1.11]{M}, it is not difficult to show that 
\begin{align}
\label{asymptoticforf(n)}
\sum_{n\leq t}d_{\alpha_N}^{\varpi}(n)=\frac{c_0(\alpha_N, \varpi)}{\Gamma(\alpha_N)}t(\log N)^{\alpha_N-1}\bigg(1+O\bigg(\frac{\log\log N}{|R(N)|\log N}\bigg)\bigg)+O\bigg(\frac{N\log\log N}{\log N}\bigg),
\end{align}
for any $t\in [N/\log N, N]$, if $N$ is large enough, where $\Gamma(z)$ stands for the Gamma function and 
\[c_0(\alpha_N, \varpi):=\left\{ \begin{array}{ll}
         \prod_{p}\left(1-\frac{1}{p}\right)^{\alpha_N}\left(1+\frac{\alpha_N}{p-1}\right)& \mbox{if $\varpi(n)=\omega(n)$};\\
\prod_{p}\left(1-\frac{1}{p}\right)^{\alpha_N}\left(1-\frac{\alpha_N}{p}\right)^{-1} & \mbox{if $\varpi(n)=\Omega(n)$}.\end{array} \right. \] 
 
It is easy to verify that
\begin{align}
\label{asymptc0gamma}
c_0(\alpha_N, \varpi)=1+O\bigg(\frac{1}{|R(N)|}\bigg)=\Gamma(\alpha_N),
\end{align}
if $N$ is large enough (see \cite[Appendix C]{MV} for basic results on the Gamma function).

By Corollary \ref{lemmapartialsumdiv}, we certainly have
\begin{align*}
\sum_{n\leq N/\log N}d_{\alpha_N}^{\varpi}(n)\phi\left(\frac{n}{N}\right)\ll \sum_{n\leq N/\log N}\bigg(1+\frac{1}{|R(N)|}\bigg)^{\varpi(n)}\ll \frac{N}{\log N}(\log N)^{1/|R(N)|}\ll \frac{N}{\log N}.
\end{align*}
This, together with partial summation from \eqref{asymptoticforf(n)} applied to the remaining part of the sum, leads to 
\begin{align*}
\sum_{n\leq N}d_{\alpha_N}^{\varpi}(n)\phi\left(\frac{n}{N}\right)=\frac{c_0(\alpha_N, \varpi)}{\Gamma(\alpha_N)}JNe^{\frac{\log\log N}{R(N)}}+O\bigg(\frac{N\log\log N}{\log N}\bigg),
\end{align*}
where $J:=\int_{0}^{1}\phi(t)dt\in [1/2, 1]$ and we made use of \eqref{asymptc0gamma} to simplify the error term.

Applying this asymptotic estimate with length of the sum $N/p$ in place of $N$, we find
\begin{align*}
\sum_{\substack{n\leq N\\ p|n}}d_{\alpha_N}^{\varpi}(n)\phi\left(\frac{n}{N}\right)&= \alpha_N\sum_{\substack{k\leq N/p\\ p\nmid k}}d_{\alpha_N}^{\varpi}(k)\phi\left(\frac{pk}{N}\right)+\sum_{k\leq N/p^2}d_{\alpha_N}^{\varpi}(kp^2)\phi\left(\frac{kp^2}{N}\right)\\
&=\alpha_N\sum_{\substack{k\leq N/p}}d_{\alpha_N}^{\varpi}(k)\phi\left(\frac{pk}{N}\right)+O\bigg(\sum_{\substack{k\leq N/p^2}}d_{1+1/|R(N)|}^{\varpi}(k)\bigg)\\
&=\frac{c_0(\alpha_N, \varpi)}{\Gamma(\alpha_N)}\frac{JN\alpha_N}{p}e^{\frac{\log\log(N/p)}{R(N)}}+O\bigg(\frac{N\log\log N}{p\log N}+\frac{N}{p^2}\bigg),
\end{align*}
where we used $\varpi(pk)\leq \varpi(k)+1$ and Corollary \ref{lemmapartialsumdiv} to handle the error term contribution.

The collection of the above estimates, taking into account of the identity \eqref{mainpropc_q} for the Ramanujan sums, makes the sum over $n$ in the statement equals to
\begin{align}
\label{asymptestimatedivtwisted}
\frac{c_0(\alpha_N, \varpi)}{\Gamma(\alpha_N)}JNe^{\frac{\log\log N}{R(N)}}( \alpha_N e^{\frac{\log(1-\frac{\log p}{\log N})}{R(N)}}-1)+O\bigg(\frac{N\log\log N}{\log N}+\frac{N}{p}\bigg).
\end{align}
By Taylor expansion and thanks to \eqref{asymptc0gamma}, one has
\begin{align*}
\alpha_Ne^{\frac{\log(1-\frac{\log p}{\log N})}{R(N)}}-1&=\bigg(1+\frac{1}{R(N)}\bigg)\bigg(1+\frac{\log(1-\frac{\log p}{\log N})}{R(N)}+O\bigg(\frac{1}{R(N)^2}\bigg)\bigg)-1\\
&=\frac{1+\log(1-\frac{\log p}{\log N})}{R(N)}+O\bigg(\frac{1}{R(N)^2}\bigg)
\end{align*}
and 
\begin{align*}
\frac{c_0(\alpha_N,\varpi)}{\Gamma(\alpha_N)}e^{\frac{\log\log N}{R(N)}}=\bigg(1+O\bigg(\frac{1}{|R(N)|}\bigg)\bigg)\bigg(1+O\bigg(\frac{\log\log N}{|R(N)|}\bigg)\bigg)=1+O\bigg(\frac{\log\log N}{|R(N)|}\bigg).
\end{align*}
Inserting the above estimates into \eqref{asymptestimatedivtwisted}, we see that the double sum in the statement is 
\begin{align*}
&=\bigg(\frac{JN}{R(N)}\sum_{\substack{2N/Q<p\leq R}}\frac{1+\log(1-\frac{\log p}{\log N})}{p}+O\bigg(\frac{N\log\log N}{R(N)^2}\bigg)\bigg)\bigg(1+O\bigg(\frac{\log\log N}{|R(N)|}\bigg)\bigg)\\
&+O\bigg(\frac{N(\log\log N)^2}{\log N}+N\sum_{2N/Q<p\leq R}\frac{1}{p^2}\bigg)\\
&\gg \frac{N}{|R(N)|}\log\bigg(\frac{\log R}{\log(2N/Q)}\bigg)+O\bigg(\frac{Q}{\log(2N/Q)}\bigg),
\end{align*}
by Mertens' theorem, by taking $C$ and $N$ large enough with respect to $\delta$ and thanks to our assumption on $|R(N)|,$ from which we get the thesis on our range of $Q$.
\end{proof}
\section{Proof of Proposition \ref{proplowerboundint1}}
By restricting the integral in the statement over minor arcs of the form $(1/q-1/KqQ,1/q+1/KqQ)$, for positive integers $q$ in the range $Q/(2M^2)<q\leq Q/M^2$, where $M$ is a large positive constant to be chosen later, we can lower bound it with
\begin{align*}
\sum_{Q/(2M^2)<q\leq Q/M^2}\int_{-1/KqQ}^{1/KqQ}\bigg|\sum_{\substack{n\leq N}} e(n/q)e(n\theta) \bigg|^2 d\theta.
\end{align*}
Since, by definition of minor arcs, $q>KQ_0$ and by \eqref{eq0} $Q_0\leq Q/K^2$, we require $K>2M^2$, say. Moreover, we remind that $K$, and thus $M$, are absolute constants here.
By partial summation it is easy to verify that
\begin{align*}
\bigg|\sum_{\substack{1\leq n\leq N}} e(n/q)e(n\theta)\bigg|= \bigg|\frac{e^{2\pi i(N+1)/q}-e^{2 \pi i/q}}{e^{2\pi i/q}-1}\bigg|+O\bigg(\frac{N}{Q}\bigg).
\end{align*}
We deduce that
\begin{align*}
\int_{\mathfrak{m}(K,Q_0,Q)}|{\mathcal S}_{1}(\theta)|^2d\theta &\geq \sum_{Q/(2M^2)<q\leq Q/M^2}\frac{2}{KqQ}\bigg|\frac{e^{2\pi i(N+1)/q}-e^{2 \pi i/q}}{e^{2\pi i/q}-1}\bigg|^2\\
&+O\bigg(\frac{N^2}{Q^3}+\frac{N}{Q}\sum_{Q/(2M^2)<q\leq Q/M^2}\frac{1}{qQ}\bigg|\frac{e^{2\pi i(N+1)/q}-e^{2 \pi i/q}}{e^{2\pi i/q}-1}\bigg|\bigg)\\
&\gg\sum_{Q/(2M^2)<q\leq Q/M^2}\frac{q}{Q}|e^{2\pi i(N+1)/q}-e^{2 \pi i/q}|^2+O\bigg(\frac{N^2}{Q^3}+\frac{N}{Q}\bigg),
\end{align*}
by expanding 
\begin{align*}
e^{2\pi i/q}-1=\frac{2\pi i}{q}+O\bigg(\frac{1}{q^2}\bigg)\asymp \frac{1}{q}.
\end{align*}
Notice that
$$|e^{2\pi i(N+1)/q}-e^{2 \pi i/q}|^2=2-2\Re(e^{2\pi iN/q}).$$
Therefore, to conclude, we only have to produce some saving on the size of the partial sum of $\Re(e^{2\pi iN/q})$ over the interval $I:=[Q/(2M^2)<q\leq Q/M^2]$ compared to its length.
Once done that, we immediately deduce that
\begin{align*}
\int_{\mathfrak{m}(K,Q_0,Q)}|{\mathcal S}_{1}(\theta)|^2d\theta \gg Q+O\bigg(\frac{N^2}{Q^3}+\frac{N}{Q}\bigg), 
\end{align*}
where the term $Q$ dominates whenever $Q\geq c\sqrt{N}$, for a suitable absolute constant $c>0$.
To this aim, we apply the van der Corput's inequality (see e.g. \cite[Ch. I, Theorem 6.5]{T}) to the function
$f_N(t):=N/t,$ for which $f_N(t)\in C^2(I)$ with $f''_N(t)\asymp NM^6 /Q^3$, for $t\in I$. We thus get
\begin{align*}
\bigg|\sum_{q\in I}\Re(e^{2\pi i f_N(q)})\bigg| \ll \frac{Q}{M^3},
\end{align*}
for any $M^{8/3}N^{1/3}\leq Q\leq N$, if we take $N$ sufficiently large, from which the thesis follows, by taking $M$ large enough.
\section{Proof of Proposition \ref{propL2integralofomega}}
Let $K$ be a large constant, $Q_0$ and $Q$ be real numbers satisfying \eqref{eq0}.
\subsection{Large values of $Q$}
By isolating the constant term $Z:=\log\log N$ and expanding the square out, we have
\begin{align*}
\int_{\mathfrak{m}}|{\mathcal S}_{\varpi}(\theta)|^2d\theta&\geq \int_{\mathfrak{m}}|{\mathcal S}_{\varpi-Z}(\theta)|^2d\theta+\int_{\mathfrak{m}}|{\mathcal S}_{Z}(\theta)|^2d\theta-2\int_{\mathfrak{m}}|{\mathcal S}_{\varpi-Z}(\theta){\mathcal S}_{Z}(\theta)|d\theta\\
&\geq \int_{\mathfrak{m}}|{\mathcal S}_{Z}(\theta)|^2d\theta-2\sqrt{\int_{\mathfrak{m}}|{\mathcal S}_{\varpi-Z}(\theta)|^2d\theta\int_{\mathfrak{m}}|{\mathcal S}_Z(\theta)|^2d\theta},\nonumber
\end{align*}
by an application of Cauchy--Schwarz's inequality. 
By completing the integral $\int_{\mathfrak{m}}|{\mathcal S}_{\varpi-Z}(\theta)|^2d\theta$ to the whole circle and using Parseval's identity followed by an application of the upper bound \eqref{variancevarpi} on the second centred moment of $\varpi(n)$, we find it is $\ll N\log\log N$.
Since from Propositions \ref{lowerboundint1} and \ref{propupperbounds} a) we know that $\int_{\mathfrak{m}}|{\mathcal S}_{1}(\theta)|^2d\theta\asymp Q$, on a wide range of $Q$, we also in particular have
\begin{align*}
\int_{\mathfrak{m}}|{\mathcal S}_{Z}(\theta)|^2d\theta\asymp Q(\log\log N)^2,
\end{align*}
whenever e.g. $Q\geq c N/\log\log N$, for any fixed constant $c>0$. By then choosing $c$ suitably large, we get the lower bound \eqref{sizeintomega} on such range of $Q$. 
\subsection{Small values of $Q$}
Assume now $N^{1/2+\delta}\leq Q< cN/\log\log N$, with $c$ as in the previous subsection, and $KQ_0<R$, where $R:=N^{1/2-\delta/2}$, for a small $\delta>0$. Let $g(r)$ be the characteristic function of the set of prime numbers smaller than $R$. We apply Proposition \ref{newprop11} with such sets of minor arcs and functions $g(r)$ and $f(n)=\varpi(n)$.
\begin{rmk}
In order to successfully apply Proposition \ref{newprop11}, as a rule of thumb, we might think of $g(r)$ as an approximation of the Dirichlet convolution $f\ast \mu(r)$, where $\mu(r)$ is the M\"{o}bius function. This motivates our choice of $g$, since for any $n\leq N$ we either have $g\ast 1(n)=\omega(n)$ or $g\ast 1(n)=\omega(n)-1$, with $\omega(n)\approx \log\log N\approx \Omega(n)$, for most of the integers $n\leq N$, by \eqref{meanvarpi}. 
\end{rmk}
With the notations introduced in Proposition \ref{newprop11}, we have
\begin{align}
\label{denomestimate1}
\int_{\frak m} |\mathcal{G}(\theta)|^2d\theta\ll N\log\bigg(\frac{\log N}{\log(2N/Q)}\bigg),
\end{align}
which follows from Proposition \ref{propupperbounds} b), on our range of $Q$.

Next, by \eqref{lowerboundintminarcs2}, with $f(n)=\varpi(n)$, the integral $\int_{\frak m} {\mathcal S}_{f}(\theta)\overline{\mathcal{G}(\theta)}d\theta$ is
\begin{align}
\label{numeratorshape}
&=\sum_{n\leq N}\varpi(n)\bigg(\sum_{\substack{p|n\\ p\le R}} 1\bigg)\phi\left(\frac{n}{N}\right)\\
&-N\sum_{q\leq KQ_0}\int_{-K/qQ}^{K/qQ}\bigg(\sum_{n\leq N}\varpi(n)c_q(n)e(n\beta)\bigg)\bigg(\frac{\textbf{1}_{q>2,\ \text{prime}}}{q}+\sum_{\substack{p\leq R}}\frac{\textbf{1}_{q=1}}{p}\bigg)\hat{\phi}(\beta N)d\beta\nonumber\\
&+O(N^{1-\delta}),\nonumber
\end{align}
if $N$ is large enough with respect to $\delta$, where we trivially estimated the error term using the bound \eqref{maxsizevarpi} on the maximal size of $\varpi(n)$ and our hypotheses on $Q_0, Q$ and $R$.
The second expression in \eqref{numeratorshape} equals
\begin{align}
\label{startingpointintest1}
&-N\sum_{n\leq N}\varpi(n)\sum_{\substack{p\leq R}}\frac{1}{p}\int_{-K/Q}^{K/Q}e(n\beta)\hat{\phi}(\beta N)d\beta\\
\label{startingpointintest2}
&-N\sum_{\substack{2\leq q\leq KQ_0\\ q\ \text{prime}}}\frac{1}{q}\sum_{n\leq N}\varpi(n)c_q(n)\int_{-K/qQ}^{K/qQ}e(n\beta)\hat{\phi}(\beta N)d\beta.
\end{align}
By changing variable and since $\phi(t)$ belongs to the Fourier class $\mathcal{F}$ as in Proposition \ref{newprop11}, one has
\begin{align}
\label{smoothint1}
N\int_{-K/Q}^{K/Q}e(n\beta)\hat{\phi}(\beta N)d\beta&=\phi\left(\frac{n}{N}\right)+O\bigg(\int_{KN/Q}^{+\infty}\hat{\phi}(u)du+\int_{-\infty}^{-KN/Q}\hat{\phi}(u)du \bigg)\\
&=\phi\left(\frac{n}{N}\right)+O\bigg(\frac{Q^4}{N^4}\bigg),\nonumber
\end{align}
where we remind that $Q<cN/\log\log N$. Thus, by the asymptotic expansion \eqref{meanvarpi} for the partial sum of $\varpi(n)$ and Mertens' theorem,  \eqref{startingpointintest1} equals
\begin{align}
\label{startingpointintest11}
-\sum_{\substack{p\leq R}}\frac{1}{p}\sum_{n\leq N}\varpi(n)\phi\left(\frac{n}{N}\right)+O\bigg(\frac{Q}{\log\log N}\bigg).
\end{align}
We now split the sum over $q$ in \eqref{startingpointintest2} into two parts according to whether $q\leq 2N/Q$ or $q>2N/Q$. In the second case, since $\hat{\phi}(\xi)$ is bounded, we find
\begin{align}
\label{trivialestimateint}
N\int_{-K/qQ}^{K/qQ}e(n\beta)\hat{\phi}(\beta N)d\beta=\int_{-KN/qQ}^{KN/qQ}e\left(\frac{nu}{N}\right)\hat{\phi}(u)du\ll \frac{N}{qQ},
\end{align}
from which we deduce that the contribution in \eqref{startingpointintest2} from the primes $q>2N/Q$ is 
\begin{align}
\label{contributionsecondset}
\ll \frac{N}{Q}\sum_{\substack{q>2N/Q\\ q\ \text{prime}}}\frac{1}{q^2}\sum_{n\leq N}\varpi(n)|c_q(n)|\ll \frac{N^2\log\log N}{Q}\sum_{\substack{q>2N/Q\\ q\ \text{prime}}}\frac{1}{q^2}\ll \frac{N\log\log N}{\log(2N/Q)}.
\end{align}
On the other hand, for values of $q\leq 2N/Q$, by changing variable and by definition of $\phi(t)$, we can rewrite the integral $\int_{-KN/qQ}^{KN/qQ}e(nu/N)\hat{\phi}(u)du$ as
\begin{align}
\label{asymtoticint}
\phi\left(\frac{n}{N}\right)+\int_{KN/qQ}^{+\infty}e\left(\frac{nu}{N}\right)\hat{\phi}(u)du+\int_{-\infty}^{-KN/qQ}e\left(\frac{nu}{N}\right)\hat{\phi}(u)du=\phi\left(\frac{n}{N}\right)+O\bigg(\frac{qQ}{N}\bigg),
\end{align}
from which we may deduce that the contribution in \eqref{startingpointintest2} coming from those primes is
\begin{align}
\label{smallcontribution}
&-\sum_{\substack{2\leq q\leq 2N/Q\\ q\ \text{prime}}}\frac{1}{q}\sum_{n\leq N}\varpi(n)c_q(n)\phi\left(\frac{n}{N}\right)+O\bigg(\frac{N\log\log N}{\log(2N/Q)}\bigg).
\end{align}
Collecting together \eqref{startingpointintest11}, \eqref{smallcontribution} and previous observations and thanks to the identity \eqref{mainpropc_q} for the Ramanujan sums, we see that \eqref{numeratorshape} equals to 
\begin{align*}
&\sum_{\substack{2N/Q<p\leq R}}\frac{1}{p}\sum_{n\leq N}\varpi(n)c_p(n)\phi\left(\frac{n}{N}\right)+O\bigg(\frac{N\log\log N}{\log(2N/Q)}\bigg),
\end{align*}
if $N$ is large enough with respect to $\delta$, where a lower bound for the size of the above sum has already been given in Lemma \ref{lempartialsumsmoothweight}. 
Overall, we have thus found that
\begin{align*}
\int_{\frak m} {\mathcal S}_{f}(\theta)\overline{\mathcal{G}(\theta)}d\theta\gg N\log\bigg(\frac{\log R}{\log(2N/Q)}\bigg),
\end{align*}
in the range $N^{1/2+\delta}\leq Q\leq cN/\log\log N$, which, together with the upper bound \eqref{denomestimate1} for the integral $\int_{\frak m} |\mathcal{G}(\theta)|^2d\theta$, concludes the proof of the lower bound \eqref{sizeintomega} for the integral $\int_{\frak m} |{\mathcal S}_{\varpi}(\theta)|^2 d\theta$, via an application of Proposition \ref{newprop11},  whenever $N$ is suitably large with respect to $\delta$. Indeed, to rewrite the result as in the statement of Proposition \ref{propL2integralofomega} we appeal to the following lemma.
\begin{lem}
\label{lemdeltadep}
For any $\delta$ small enough and $N$ sufficiently large with respect to $\delta$, we have
\begin{align*}
\log\bigg(\frac{\log R}{\log(2N/Q)}\bigg)\geq \delta\log\bigg(\frac{\log N}{\log(2N/Q)}\bigg).
\end{align*} 
\end{lem}
\begin{proof}
The aimed inequality is equivalent to
\begin{align*}
\bigg(\frac{1}{2}-\frac{\delta}{2}\bigg)\bigg(\frac{\log N}{\log(2N/Q)}\bigg)^{1-\delta}\geq 1,
\end{align*}
which is satisfied when in particular 
\begin{align*}
\left(\frac{1}{2}-\frac{\delta}{2}\right)\geq \left(\frac{1}{2}-\delta+O(\delta^2)\right)^{1-\delta}
\end{align*}
and $N$ is sufficiently large with respect to $\delta$. The above in turn is equivalent to
\begin{align*}
\frac{1+\frac{\delta}{\log2}+O(\delta^2)}{1+\frac{2\delta}{\log 2}+O(\delta^2)}\leq 1-\delta.
\end{align*}
Since the left-hand side above equals to $1-\delta/\log 2+O(\delta^2),$ the thesis immediately follows if $\delta$ is taken small enough.
\end{proof}
\section{Proof of Proposition \ref{proplowerboundintomega}}
Let $K$ be a large constant, $Q_0$ and $Q$ be real numbers satisfying \eqref{eq0}.
Moreover, let $C\log\log N\leq |R(N)|\leq N^{\delta/12}$, with $C$ as in Lemma \ref{lempartialsumdivtwisted}. 
\subsection{Large values of $Q$}
By isolating the constant term $1$ and expanding the square out, we have
\begin{align*}
\int_{\mathfrak{m}}|{\mathcal S}_{d_{\alpha_N}^{\varpi}}(\theta)|^2d\theta&\geq \int_{\mathfrak{m}}|{\mathcal S}_{d_{\alpha_N}^{\varpi}-1}(\theta)|^2d\theta+\int_{\mathfrak{m}}|{\mathcal S}_{1}(\theta)|^2d\theta-2\int_{\mathfrak{m}}|{\mathcal S}_{d_{\alpha_N}^{\varpi}-1}(\theta){\mathcal S}_{1}(\theta)|d\theta\\
&\geq \int_{\mathfrak{m}}|{\mathcal S}_{1}(\theta)|^2d\theta-2\sqrt{\int_{\mathfrak{m}}|{\mathcal S}_{d_{\alpha_N}^{\varpi}-1}(\theta)|^2d\theta\int_{\mathfrak{m}}|{\mathcal S}_1(\theta)|^2d\theta},\nonumber
\end{align*}
by an application of Cauchy--Schwarz's inequality. 
The estimate of the integral $\int_{\mathfrak{m}}|{\mathcal S}_{d_{\alpha_N}^{\varpi}-1}(\theta)|^2d\theta$ has already been performed in subsect. $4.3$, where we found (see Eq. \eqref{boundexpsumdiv}):
\begin{align*}
\int_{\mathfrak{m}}|{\mathcal S}_{d_{\alpha_N}^{\varpi}-1}(\theta)|^2d\theta\ll\frac{1}{R(N)^2}\int_{\mathfrak{m}}|{\mathcal S}_{\varpi}(\theta)|^2d\theta+\frac{N(\log\log N)^3}{R(N)^4}+\frac{N}{R(N)^2\log N}.
\end{align*}
By Propositions \ref{proplowerboundint1} and \ref{propupperbounds} a), which together give $\int_{\mathfrak{m}}|{\mathcal S}_{1}(\theta)|^2d\theta\asymp Q$, by Proposition \ref{propupperbounds} b), which shows that 
$$\int_{\frak m} |{\mathcal S}_{\varpi}(\theta)|^2 d\theta\ll Q(\log\log N)^2+N\log\bigg(\frac{\log N}{\log (2N/Q)}\bigg),$$ and by the above considerations, we may deduce the lower bound \eqref{lowerboundintomega1.5}, at least when $Q\geq cN(\log\log N)/R(N)^2$, for $c$ a suitable positive constant, by taking $N$ large enough and possibly replacing $C$ with a larger value.
\subsection{Small values of $Q$}
Let us now assume $N^{1/2+\delta}\leq Q<cN(\log\log N)/R(N)^2$ and $KQ_0<R$, where $R:=N^{1/2-\delta/2}$, for a small $\delta>0$. Let $g(r)$ be the characteristic function of the set of prime numbers smaller than $R$. We apply Proposition \ref{newprop11} with such sets of minor arcs and functions $g(r)$ and $f(n)=d_{\alpha_N}^{\varpi}(n)$. With the notations introduced there, we again have
\begin{align}
\label{denomestimate2}
\int_{\frak m} |\mathcal{G}(\theta)|^2d\theta\ll N\log\bigg(\frac{\log N}{\log(2N/Q)}\bigg),
\end{align}
which follows from Proposition \ref{propupperbounds} b), since by assumption on $|R(N)|$ we always at least have $Q\ll N/\log\log N$.

Next, by \eqref{lowerboundintminarcs2}, with $f(n)=d_{\alpha_N}^{\varpi}(n)$, the integral $\int_{\frak m} {\mathcal S}_{f}(\theta)\overline{\mathcal{G}(\theta)}d\theta$ is
\begin{align}
\label{numeratorshape2}
&=\sum_{n\leq N}d_{\alpha_N}^{\varpi}(n)\bigg(\sum_{\substack{p|n\\ p\le R}} 1\bigg)\phi\left(\frac{n}{N}\right)\\
&-N\sum_{q\leq KQ_0}\int_{-K/qQ}^{K/qQ}\bigg(\sum_{n\leq N}d_{\alpha_N}^{\varpi}(n)c_q(n)e(n\beta)\bigg)\bigg(\frac{\textbf{1}_{q>2,\ \text{prime}}}{q}+\sum_{\substack{p\leq R}}\frac{\textbf{1}_{q=1}}{p}\bigg)\hat{\phi}(\beta N)d\beta\nonumber\\
&+O(N^{1-\delta}),\nonumber
\end{align}
if $N$ is large enough with respect to $\delta$, where we trivially estimated the error term using Corollary \ref{lemmapartialsumdiv} and our hypotheses on $Q_0, Q$ and $R$.

The second expression in the above displayed equation equals
\begin{align}
\label{startingpointintest111}
&-N\sum_{n\leq N}d_{\alpha_N}^{\varpi}(n)\sum_{\substack{p\leq R}}\frac{1}{p}\int_{-K/Q}^{K/Q}e(n\beta)\hat{\phi}(\beta N)d\beta\\
\label{startingpointintest222}
&-N\sum_{\substack{2\leq q\leq KQ_0\\ q\ \text{prime}}}\frac{1}{q}\sum_{n\leq N}d_{\alpha_N}^{\varpi}(n)c_q(n)\int_{-K/qQ}^{K/qQ}e(n\beta)\hat{\phi}(\beta N)d\beta.
\end{align}
By the second identity in \eqref{smoothint1} for $N\int_{-K/Q}^{K/Q}e(n\beta)\hat{\phi}(\beta N)d\beta$, we see that \eqref{startingpointintest111} is
\begin{align}
\label{firstexpressionestimate}
&=-\sum_{\substack{p\leq R}}\frac{1}{p}\sum_{n\leq N}d_{\alpha_N}^{\varpi}(n)\phi(n/N)+O\bigg(\frac{Q}{R(N)^2} \bigg),
\end{align}
where we used Corollary \ref{lemmapartialsumdiv} and our hypothesis on $Q$ to estimate the error term.

We now split the sum over $q$ in \eqref{startingpointintest222} into two parts according to whether $q\leq 2N/Q$ or $q>2N/Q$. The term corresponding to the second set of primes equals to
\begin{align}
\label{secondexpression}
&-\frac{N}{R(N)}\sum_{\substack{2N/Q< q\leq KQ_0\\ q\ \text{prime}}}\frac{1}{q}\sum_{n\leq N}\varpi(n)c_q(n)\int_{-K/qQ}^{K/qQ}e(n\beta)\hat{\phi}(\beta N)d\beta\\
&-N\sum_{\substack{2N/Q<  q\leq KQ_0\\ q\ \text{prime}}}\frac{1}{q}\sum_{n\leq N}c_q(n)\int_{-K/qQ}^{K/qQ}e(n\beta)\hat{\phi}(\beta N)d\beta\nonumber\\
&-N\sum_{\substack{2N/Q<  q\leq KQ_0\\ q\ \text{prime}}}\frac{1}{q}\sum_{n\leq N}E(n)c_q(n)\int_{-K/qQ}^{K/qQ}e(n\beta)\hat{\phi}(\beta N)d\beta,\nonumber
\end{align}
where for the sake of readiness we defined
$E(n):=d_{\alpha_N}^{\varpi}(n)-1-\varpi(n)/R(N)$.
The sum in the first term above has already been estimated before, with the result given in \eqref{contributionsecondset}. Whence, the first expression in \eqref{secondexpression} is 
$$\ll \frac{N\log\log N}{|R(N)|\log(2N/Q)}.$$
Regarding the second term in \eqref{secondexpression},
by changing variable inside the integral and swapping integral and summation, it is 
\begin{align*}
-\sum_{\substack{2N/Q<q\leq KQ_0\\ q\ \text{prime}}}\frac{1}{q}\int_{-KN/qQ}^{KN/qQ}\sum_{\substack{n\leq N}}c_q(n)e\left(\frac{nu}{N}\right)\hat{\phi}(u)du\ll \frac{N}{Q}\sum_{\substack{2N/Q< q\leq KQ_0\\ q\ \text{prime}}}\frac{1}{q}\ll \frac{N\log\log N}{Q}\leq \sqrt{N},
\end{align*}
by Lemma \ref{lemramanujansmooth}, Mertens' theorem and taking $N$ large enough with respect to $\delta$.

Finally, the third term in \eqref{secondexpression}, by the estimate \eqref{trivialestimateint} for $N\int_{-K/qQ}^{K/qQ}e(n\beta)\hat{\phi}(\beta N)d\beta$, the identity \eqref{mainpropc_q} for the Ramanujan sums and the bound \eqref{secondmomentvarpi} on the second moment of $\varpi(n)$, is easily seen to be
\begin{align*}
\ll \frac{N^2(\log\log N)^2}{QR(N)^2}\sum_{\substack{2N/Q< q\leq KQ_0\\ q\ \text{prime}}}\frac{1}{q^2}\ll \frac{N(\log\log N)^2}{R(N)^2\log(2N/Q)}.
\end{align*}
Here, to estimate the sum over $n$ in \eqref{secondexpression} we argued as in subsect. $4.3$, by dividing the argument according to whether $|R(N)|\leq (\log N)/(\log 2)$ or not, and, in the first case, by splitting the sum over those integers $n$ such that $\varpi(n)\leq C(\log\log N)$ or the opposite holds.

Regarding the part of \eqref{startingpointintest222} corresponding to primes $q\leq 2N/Q$, we first rewrite the integral $\int_{-KN/qQ}^{KN/qQ}e(nu/N)\hat{\phi}(u)du$ as in \eqref{asymtoticint}. Afterwards, by writing $d_{\alpha_N}^{\varpi}(n)=:1+\varpi(n)/R(N)+E(n)$, using Lemma \ref{lemramanujansmooth} to handle the contribution coming from the constant function $1$ and arguing similarly as before to compute the contribution from $\varpi(n)$ and $E(n)$, we readily see that such part equals to
\begin{align*}
-\sum_{\substack{2\leq q\leq 2N/Q\\ q\ \text{prime}}}\frac{1}{q}\sum_{n\leq N}d_{\alpha_N}^{\varpi}(n)c_q(n)\phi\left(\frac{n}{N}\right)+O\bigg(\frac{N\log\log N}{|R(N)|\log(2N/Q)}\bigg).
\end{align*}
Overall, we have found that \eqref{numeratorshape2} is
\begin{align}
\label{onestepbeforeend}
\sum_{\substack{2N/Q<q\leq R\\ q\ \text{prime}}}\frac{1}{q}\sum_{n\leq N}d_{\alpha_N}^{\varpi}(n)c_q(n)\phi\left(\frac{n}{N}\right)+O\bigg(\frac{N\log\log N}{|R(N)|\log(2N/Q)}+\frac{Q}{R(N)^2}\bigg),
\end{align}
if $N$ is sufficiently large with respect to $\delta$.

We now split the argument into two parts, according to whether $|R(N)|\leq (\log\log N)^{3}$ or not. In the first case, we remind that the size of the above sum has already been estimated in Lemma \ref{lempartialsumdivtwisted}. From this, the upper bound \eqref{denomestimate2} for the integral $\int_{\frak m} |\mathcal{G}(\theta)|^2d\theta$ and taking into account of Lemma \ref{lemdeltadep}, we may deduce the lower bound \eqref{lowerboundintomega1.5} for the integral $\int_{\frak{m}} |{\mathcal S}_{d_{\alpha_N}^{\varpi}}(\theta)|^2 d\theta$ in such range of $|R(N)|$, via an application of Proposition \ref{newprop11}, if $N$ is suitably large with respect to $\delta$. 

On the other hand, when $|R(N)|>(\log\log N)^3$, we replace $d_{\alpha_N}^{\varpi}(n)$ inside \eqref{onestepbeforeend} with $1+\varpi(n)/R(N)+E(n)$. Afterwards, we estimate the error contribution coming from the constant function $1$ using partial summation from the bound \eqref{defandpropS} on the partial sum of $c_q(n)$ and trivially that from $E(N)$ thanks to our current assumption on $|R(N)|$ and arguing as before. Finally, the main contribution coming from $\varpi(n)/R(N)$ can be immediately handled by Lemma \ref{lempartialsumsmoothweight}. Combining the estimate we get, by proceeding in this way, for \eqref{numeratorshape2} together with the bound \eqref{denomestimate2} via an application of Proposition \ref{newprop11}, we may deduce the lower bound \eqref{lowerboundintomega1.5} also on this range of $|R(N)|$ and thus conclude the proof of Proposition \ref{proplowerboundintomega}.
\section{Proof of Proposition \ref{proplowerboundvariancesmooth2}}
\subsection{Large values of $Q$}
We always have 
\begin{align*}
\int_{\frak m} |{\mathcal S}_{\textbf{1}_{y-\text{smooth}}}(\theta)|^2 d\theta\geq \int_{\frak m} |{\mathcal S}_{1}(\theta)|^2 d\theta+\int_{\frak m} \bigg|\sum_{\substack{n\leq N\\ \exists p|n: p> y}} e(n\theta)\bigg|^2d\theta-2\int_{\frak m} \bigg|{\mathcal S}_{1}(\theta)\sum_{\substack{n\leq N\\ \exists p|n: p> y}} e(n\theta)\bigg| d\theta.
\end{align*}
By Parseval's identity and Mertens' theorem, the second integral on the right-hand side above is $\ll N\log u$, where $u:=(\log N)/(\log y)$. This, together with the upper bound for $\int_{\frak m} |{\mathcal S}_{1}(\theta)|^2 d\theta$ given in Proposition \ref{propupperbounds} a) and Cauchy--Schwarz's inequality, makes the third integral instead of size $\ll \sqrt{QN\log u}$. By using the lower bound \eqref{lowerboundint1} for the integral $\int_{\frak m} |{\mathcal S}_{1}(\theta)|^2 d\theta$, for values $DN\log u\leq Q\leq N$, with $D>0$ a large constant, we may deduce the lower bound \eqref{lowerboundsmoothint} on such range of $Q$.
\subsection{Small values of $Q$} Let $\delta>0$ small. Let $K$ be a large constant, $Q_0$ and $Q$ be real numbers satisfying \eqref{eq0} and such that $N^{1/2+\delta}\leq Q<DN\log u,$ with $D$ as in the previous subsection, and $\log N<Q_0\leq Q_0^{\max}:=N^{1/2-\delta}(\log N)^{17}/K$. Let $R:=N^{1/2-\delta/2}$. We keep these notations throughout the rest of this section.
\begin{rmk}
The choice of the maximal possible size of $Q_0$ only reflects the fact that, to deduce the lower bound on the variance of the $y$--smooth numbers in arithmetic progressions as in Theorem \ref{lowerboundvariancesmooth2}, we will take $Q_0=N(\log N)^{17}/Q$ in Proposition \ref{Prop5.1}.
\end{rmk}
\subsubsection{Case $y$ small}
Let $\sqrt{N}\leq y\leq N^{1-\delta/8}$.
Let $g(r)$ be the indicator of the prime numbers $r\in [Q_0^{\max},R]$. We apply Proposition \ref{newprop10} with functions $f(n)=\textbf{1}_{y-\text{smooth}}(n)$ and $g(r)$ as above. 
\begin{rmk}
The choice of $g$ here has been inspired by the fact that the Dirichlet convolution $\textbf{1}_{y-\text{smooth}}\ast \mu (n)$ equals $\textbf{1}_{\text{primes}\ \in (y, N]}(n)$.
\end{rmk}
With notations as in Proposition \ref{newprop10}, by Parseval's identity, we have
\begin{align}
\label{applicationCS2}
\int_{\frak m} |\mathcal{G}(\theta)|^2d\theta\leq \sum_{n\leq N}\bigg(\sum_{\substack{p|n\\ Q_0^{\max}<p\leq R}}1\bigg)^2\leq \sum_{Q_0^{\max}<p\leq R}\frac{N}{p}+\sum_{\substack{Q_0^{\max}<p_1,p_2\leq R\\ p_1\neq p_2}} \frac{N}{p_1p_2}\ll_\delta N\nonumber,
\end{align}
by expanding the square out and swapping summations.

Let $W:=\min\{N/y, R\}$ and $Z:=\max\{KQ_0^{\max}, N/y\}.$
By \eqref{eq81}, with $f(n)=\textbf{1}_{y-\text{smooth}}(n)$, and employing the first part of Lemma \ref{lemsumsmoothtwisted}, we get
\begin{align*}
\int_{\frak m} |{\mathcal S}_{f}(\theta)\mathcal{G}(\theta) | d\theta\gg \frac{N}{\log N}\sum_{\substack{KQ_0^{\max}<q\leq W\\ q\ \text{prime}}}\frac{\log q}{q}+N\log u\sum_{\substack{Z<q\leq R\\ q\ \text{prime}}}\frac{1}{q}+O_\delta(N^{1-\delta/11})\gg_\delta N,
\end{align*}
by Mertens' theorem, if $N$ is large with respect to $\delta$.
This concludes the proof of Proposition \ref{proplowerboundvariancesmooth2} when $\sqrt{N}\leq y\leq N^{1-\delta/8}$ via the application of Proposition \ref{newprop10} and the results just proved.
\subsubsection{Case $y$ large} Let us now consider $N^{1-\delta/8}<y\leq N/C,$ where $C$ is as in Lemma \ref{lemsumsmoothtwisted}. Let $g$ be a multiplicative function supported on the squarefree numbers and given on the primes by
\begin{equation*}
g(p)= \left\{ \begin{array}{ll}
         1 & \mbox{if $N/y<p\leq R$};\\
         0 & \mbox{otherwise}.\end{array} \right. 
\end{equation*} 
We again apply Proposition \ref{newprop10} with functions $f(n)=\textbf{1}_{y-\text{smooth}}(n)$ and $g(r)$ as above. 
\begin{rmk}
From the work in subsubsect. 9.2.1, it is clear that we cannot make use of the same type of $g$ even when $y$ is very close to $N$. For, we would always have
\begin{align*}
\int_{\frak m} |\mathcal{G}(\theta)|^2d\theta\ll N\max\left\{\sum_{p\in \text{Supp}(g)\cap [KQ_0,R]}\frac{1}{p},\bigg(\sum_{p\in \text{Supp}(g)\cap [KQ_0,R]}\frac{1}{p}\bigg)^2\right\},
\end{align*} 
whereas by \eqref{eq81} and Lemma \ref{lemsumsmoothtwisted} we would always also have
\begin{align*}
\int_{\frak m} |{\mathcal S}_{f}(\theta)\mathcal{G}(\theta) | d\theta\gg N\log u \sum_{p\in \text{Supp}(g)\cap [KQ_0,R]}\frac{1}{p},
\end{align*}
which are not of comparable size, whenever $u$ is close to $1$. For such values of $y$, we then opted for a multiplicative function $g$ with the right logarithmic density, suggested to us from the second part of Lemma \ref{lemsumsmoothtwisted} and the following computations. 
\end{rmk}
By Parseval's identity, we have
\begin{align*}
\int_{\frak m} |\mathcal{G}(\theta)|^2d\theta\leq\sum_{n\leq N}\bigg(\sum_{\substack{r|n\\ r\leq R\\ p|r\Rightarrow N/y<p\leq R}}1\bigg)^2\leq \sum_{\substack{r_1,r_2\leq R\\ p|r_1,r_2\Rightarrow N/y<p\leq R}}\sum_{\substack{n\leq N\\ [r_1,r_2]|n}}1\leq N\sum_{\substack{r_1,r_2\leq R\\ p|r_1,r_2\Rightarrow N/y<p\leq R}}\frac{1}{[r_1,r_2]},
\end{align*}
by expanding the square and swapping summations. By using a manipulation employed in a work of Dress, Iwaniec and Tenenbaum (see \cite[Eq. 1]{DIT}) we can rewrite the last sum above as
\begin{align*}
\sum_{\substack{r_1,r_2\leq R\\ p|r_1,r_2\Rightarrow N/y<p\leq R}}\frac{1}{r_1r_2}\sum_{d|r_1,r_2}\varphi(d)&\leq \sum_{\substack{d\leq R\\ p|d\Rightarrow N/y<p\leq R}}\frac{\varphi(d)}{d^2}\bigg(\sum_{\substack{k\leq R\\ p|k\Rightarrow N/y<p\leq R}}\frac{1}{k}\bigg)^2\leq\bigg(\sum_{\substack{k\leq R\\ p|k\Rightarrow N/y<p\leq R}}\frac{1}{k}\bigg)^3.
\end{align*}
Since the last sum in the above displayed equation is
\begin{align*}
\ll \prod_{N/y<p\leq R}\bigg(1+\frac{1}{p}\bigg)\ll\exp\bigg(\sum_{N/y<p\leq R}\frac{1}{p}\bigg)\ll \frac{\log R}{\log(N/y)}\ll \frac{1}{u-1},
\end{align*}
thanks to Lemma \ref{Rankinestimate0} and Mertens' theorem, we deduce that 
\begin{equation}
\label{denominatorestimatesmooth}
\int_{\frak m} |\mathcal{G}(\theta)|^2d\theta\ll \frac{N}{(u-1)^3}.
\end{equation}
We note that
\begin{align*}
\sum_{\substack{r\leq R\\ q|r\\ \mu^2(r)=1}}\frac{g(r)}{r}=\frac{g(q)}{q}\sum_{\substack{k\leq R/q\\ (q,k)=1\\ \mu^2(k)=1}}\frac{g(k)}{k}\geq \frac{g(q)}{q}\prod_{p|q}\bigg(1+\frac{g(p)}{p}\bigg)^{-1}\sum_{\substack{k\leq R/q\\ \mu^2(k)=1}}\frac{g(k)}{k}=:\frac{h(q)}{q}\sum_{\substack{k\leq R/q\\ \mu^2(k)=1}}\frac{g(k)}{k},
\end{align*}
where we observe that $h(q)$ is a positive multiplicative function.
Supposing $q\leq N^{1/2-3\delta/4}$, using Lemma \ref{Rankinestimate0} and Mertens' theorem, we have 
\begin{align*}
\sum_{\substack{k\leq R/q\\ \mu^2(k)=1}}\frac{g(k)}{k}\gg \exp\bigg(\sum_{N/y<p\leq N^{\delta/4}}\frac{1}{p}\bigg)\gg \frac{\log N^{\delta/4}}{\log(N/y)}\gg_\delta \frac{1}{u-1}.
\end{align*}
By \eqref{eq81}, with $f(n)=\textbf{1}_{y-\text{smooth}}(n)$, and employing the second part of Lemma \ref{lemsumsmoothtwisted} after restricting the summation over $q$ on those integers $N^{1/2-5\delta/6}<q\leq N^{1/2-3\delta/4}$, we find
\begin{align*}
\int_{\frak m} |{\mathcal S}_{f}(\theta)\mathcal{G}(\theta) | d\theta\gg_\delta \frac{N\log u}{u-1}\sum_{N^{1/2-5\delta/6}<q\leq N^{1/2-3\delta/4}}\frac{h(q)}{q},
\end{align*}
if $N$ is large enough also with respect to $\delta$.

Let $\mathcal{P}=\prod_{p\leq N/y}p$. For any integer $k\geq 0$, we let
\begin{align*}
&S_1(k):=\bigg(\sum_{\substack{2^k N^{1/2-5\delta/6}<q\leq 2^{k+1}N^{1/2-5\delta/6}\\ (q,\mathcal{P})=1\\ \mu^2(q)=1}}1 \bigg)^2\\
&S_2(k):=\sum_{\substack{2^k N^{1/2-5\delta/6}<q\leq 2^{k+1}N^{1/2-5\delta/6}\\ \\ (q,\mathcal{P})=1\\ \mu^2(q)=1}}\frac{q}{h(q)}.
\end{align*}
By dyadic subdivision, one has
\begin{align*}
\sum_{N^{1/2-5\delta/6}<q\leq N^{1/2-3\delta/4}}\frac{h(q)}{q}\geq\sum_{k=0}^{\frac{\delta \log N}{12\log 2}-1}\sum_{2^k N^{1/2-5\delta/6}<q\leq 2^{k+1}N^{1/2-5\delta/6}}\frac{h(q)}{q}\geq \sum_{k=0}^{\frac{\delta \log N}{12\log 2}-1}\frac{S_1(k)}{S_2(k)},
\end{align*}
by Cauchy--Schwarz's inequality, where we have restated the condition on the support of $q$, implicit in $h(q)$, as $\mu^2(q)=1$ and $(q,\mathcal{P})=1$. 
By the fundamental lemma of sieve theory (see e.g. \cite[Ch. I, Theorem 4.4]{T}), taking $\delta$ small enough, and Mertens' theorem, we have
$$S_1(k)\gg_{\delta} \bigg(2^k N^{1/2-5\delta/6} \frac{\varphi(\mathcal{P})}{\mathcal{P}}\bigg)^2\gg \bigg(\frac{2^k N^{1/2-5\delta/6}}{\log(N/y)}\bigg)^2.$$ 
On the other hand, by Lemma \ref{Rankinestimate0} and Mertens' theorem, we get that $S_2(k)$ is 
\begin{align*}
\leq \sum_{\substack{q\leq 2^{k+1}N^{1/2-5\delta/6}\\ (q,\mathcal{P})=1\\ \mu^2(q)=1}}\frac{2^{k+1} N^{1/2-5\delta/6}}{h(q)}\ll \frac{(2^{k} N^{1/2-5\delta/6})^2}{\log N}\prod_{N/y<p\leq N^{1/2-3\delta/4}}\bigg(1+\frac{1}{p}\bigg)\ll \frac{(2^{k} N^{1/2-5\delta/6})^2}{\log(N/y)}.
\end{align*}
Putting things together, we have proved that
\begin{align*}
\sum_{N^{1/2-5\delta/6}<q\leq N^{1/2-3\delta/4}}\frac{h(q)}{q}\gg_{\delta} \sum_{k=0}^{\frac{\delta \log N}{12\log 2}-1}\frac{1}{\log(N/y)}\gg_\delta \frac{\log N}{\log(N/y)}\geq \frac{1}{u-1}
\end{align*}
and consequently that
\begin{align*}
\int_{\frak m} |{\mathcal S}_{f}(\theta)\mathcal{G}(\theta) | d\theta\gg_{\delta} \frac{N\log u}{(u-1)^2}.
\end{align*}
This, in combination with the upper bound \eqref{denominatorestimatesmooth} for the integral $\int_{\frak m} |\mathcal{G}(\theta)|^2d\theta$ and $\log u\gg u-1$, if $\delta$ small, concludes the proof of Proposition \ref{proplowerboundvariancesmooth2} via the application of Proposition \ref{newprop10}.
\section{Deduction of Theorem \ref{varianced1}}
By Proposition \ref{Prop5.1}, we have
\begin{align}
\label{partiallowerboundvarianced1}
V(N,Q;d_1)\gg Q\int_{\frak m} |{\mathcal S}_{1}(\theta)|^2d\theta + O \Big( \frac{N^2}{Q_0}+\sum_{q\le Q } \frac{1}{q} \sum_{\substack {d|q \\ d>Q_0}} \frac{1}{\varphi(d)} \Big| \sum_{n\leq N}c_d(n)\Big|^2\bigg),
\end{align}
by choosing $K$ large and where $Q$ and $Q_0$ need to satisfy \eqref{eq0}. 

The sum in the big-Oh error term has already been estimated in \cite[Proposition 4.3]{M}, but here we are going to produce a better bound for the function $d_1(n)$. 

First of all, by \eqref{mainpropc_q}, we notice that
\begin{align*}
\sum_{n\leq N}c_d(n)=\sum_{n\leq N}\sum_{k|(n,d)}k\mu\left(\frac{d}{k}\right)=\sum_{k|d}k\mu\left(\frac{d}{k}\right)\sum_{\substack{n\leq N\\ k|n}}1=\sum_{k|d}k\mu\left(\frac{d}{k}\right)\bigg\lfloor \frac{N}{k}\bigg\rfloor=O(\sigma(d)),
\end{align*}
where $\sigma(d):=\sum_{k|d}k$ and where we used the well-known identity $\sum_{k|d}\mu(k)=0$, for any $d>1$. Therefore, we need to study the following sum:
\begin{align}
\label{sumerrorq}
\sum_{q\le Q } \frac{1}{q} \sum_{\substack {d|q \\ d>Q_0}} \frac{\sigma(d)^2}{\varphi(d)}=\sum_{\substack {Q_0<d\leq Q}} \frac{\sigma(d)^2}{\varphi(d)}\sum_{\substack{q\le Q\\ d|q}} \frac{1}{q}\ll \sum_{\substack {Q_0<d\leq Q}} \frac{\sigma(d)^2}{d\varphi(d)}\left(\log\left( \frac{Q}{d}\right)+1\right).
\end{align}
Now, let
$$S(t):=\sum_{\substack {d\leq t}} \frac{\sigma(d)^2}{d\varphi(d)}\ \ \ (t\geq 1).$$
It is not difficult to verify that the summand satisfies the hypotheses of Lemma \ref{Rankinestimate0}, from which we easily deduce that $S(t)\ll t$, for any $t\geq 1$. By partial summation, we find that the last sum in \eqref{sumerrorq} is $\ll Q$, on our range of parameters $K, Q_0$ and $Q$ satisfying \eqref{eq0}.

We employ Proposition \ref{lowerboundint1} to lower bound the integral in \eqref{partiallowerboundvarianced1} and choose $Q_0=CN^2/Q^2$, with $C>0$ a large constant, to get the thesis for any $Q$ in the range $C^{1/3}K^{2/3}N^{2/3}\leq Q\leq CN/\log N$ (remember that $Q_0$ has to satisfy \eqref{eq0}). By doing the same, but with $Q_0=N^2(\log N)/Q^2$, we get instead the thesis for any $Q$ in the range $K^{2/3}N^{2/3}(\log N)^{1/3}\leq Q\leq N$. Together, they give Theorem \ref{varianced1}, whenever $N$ is sufficiently large.
\section{Deduction of Theorem \ref{thmvariancealpha1}}
In this final section we prove the lower bound for the variance of $d_{\alpha_N}^{\varpi}(n)$ in arithmetic progressions as presented in Theorem \ref{thmvariancealpha1}. The proofs of Theorems \ref{varianceomega} and \ref{lowerboundvariancesmooth2} are similar, so they will be omitted.

By plugging the lower bound \eqref{lowerboundintomega1.5} for the integral $\int_{\frak{m}} |{\mathcal S}_{d_{\alpha_N}^{\varpi}}(\theta)|^2 d\theta$ into the lower bound expression \eqref{estimateprop1} for the variance of $f(n)=d_{\alpha_N}^{\varpi}(n)$ in arithmetic progressions, and choosing $K$ large enough, we find
\begin{align}
\label{variancelowerboundfinal}
V(N,Q; d_{\alpha_N}^{\varpi})&\gg_{\delta} \frac{QN}{R(N)^2}\log\bigg(\frac{\log N}{\log(2N/Q)}\bigg)+Q^2+O\bigg(\frac{N^2(\log N)^{14}}{Q_0}\bigg),
\end{align}
where to estimate the error term we used \cite[Proposition 4.3]{M} with $\kappa=2$, say, and Corollary \ref{lemmapartialsumdiv}. Taking $Q_0:=N R(N)^2(\log N)^{15}/Q$, which satisfies the hypotheses of Proposition \ref{proplowerboundintomega}, we get the thesis, if $N$ is large enough with respect to $\delta$.
\section*{Acknowledgements}
I am deeply indebted to my supervisor Adam J. Harper for some discussions and insightful comments that notably improved the results presented here and simplified the exposition in this paper.

\end{document}